\documentclass[10pt]{article}
\usepackage{tikz}
\usepackage[english]{babel}
\usepackage{hyperref}
\usepackage{amsmath,amsfonts,amssymb,amsthm,enumerate,mathabx}
\usepackage{latexsym}
\usepackage{amscd}
\usepackage{hyperref}
\usepackage{vmargin}
\usepackage{cite}
\setmarginsrb{23mm}{16mm}{23mm}{16mm}%
            {14mm}{8mm}{14mm}{14mm}

\numberwithin{equation}{section}

\newtheorem{definition}{Definition}[section]
\newtheorem{lemma}[definition]{Lemma}
\newtheorem{theorem}[definition]{Theorem}
\newtheorem{corollary}[definition]{Corollary}
\newtheorem{conjecture}[definition]{Conjecture}
\newtheorem{proposition}[definition]{Proposition}
\newtheorem{fact}[definition]{Fact}
\newtheorem{em-example}[definition]{Example}
\newtheorem{em-def}[definition]{Definition}        
\newtheorem{em-remark}[definition]{Remark}         
\newtheorem{em-question}[definition]{Question}
\newtheorem{question}[definition]{Question}

\newtheorem{problem}[definition]{Problem}

\newenvironment{example}{\begin{em-example} \em }{ \end{em-example}}
\newenvironment{remark}{\begin{em-remark} \em }{\end{em-remark}}

\newcommand{\R}{\mathbb R}
\newcommand{\N}{\mathbb N}
\newcommand{\C}{\mathfrak C}

\newcommand{\Q}{\mathbb Q}
\newcommand{\Z}{\mathbb Z}

\newcommand{\T}{\mathcal{T}}

\renewcommand{\hom}{\mathrm{\hom}}

\def\Hom{\mathrm{Hom}}

\def\ent{\mathrm{ent}}
\def\End{\mathrm{End}}
\def\Aut{\mathrm{Aut}}
\def\Im{\mathrm{Im}}

\global\def\lmod#1{#1\text{-}\mathrm{Mod}}

\def\Fin{\mathrm{Fin}}
\def\lFin{\mathrm{lFin}}
\def\F{\mathcal F}
\renewcommand\ker{\mathrm{Ker}}

\def\Mat{\mathrm{Mat}}
\def\A{\mathfrak A}
\def\D{\mathfrak D}

\def\F{\mathcal{F}}
\renewcommand\hom{\mathrm{Hom}}
\def\Gdim{\mathrm{G.dim}}

\def\A{\mathcal A}

\def\link{\rightsquigarrow}

\def\2link{\stackrel{{}_2}{\link}}

\def\rk{\mathrm{rk}}

\def\ent{\mathrm{ent}}
\def\End{\mathrm{End}}
\def\Aut{\mathrm{Aut}}
\def\Im{\mathrm{Im}}

\def\F{\mathcal F}

\def\L{\mathcal L}

\def\RG {R\asterisk G}

\def\loc{{\bf L}}
\def\Q{{\bf Q}}
\def\S{{\bf S}}
\def\tor{{\bf T}}
\def\K{\mathbb K }

\def\In{In}
\def\s{\mathfrak s}
\def\Out{Out}

\input xy
\xyoption{all}

\title{Algebraic entropy of amenable group actions}

\author{Simone Virili\footnote{The author was partially supported by DGI MINECO
MTM2011-28992-C02-01, MINECO MTM2014-53644-P, by the
Comissionat per Universitats i
Recerca de la Generalitat de Catalunya and by the Fondazione Cassa di Risparmio di Padova e Rovigo (Progetto di Eccellenza ``Algebraic structures and their applications'')}}

\begin{document}

\maketitle
\abstract{Let $R$ be a ring, let $G$ be an amenable group and let $\RG$ be a crossed product. The goal of this paper is to construct, starting with a suitable additive function $L$ on the category of left modules over $R$, an additive function on a subcategory of the category of left modules over $\RG$, which coincides with the whole category  if $L({}_RR) <\infty$. This construction can be performed using a dynamical invariant associated with the original  function $L$, called algebraic $L$-entropy. We apply our results to two classical problems on group rings: the Stable Finiteness and the Zero-Divisors Conjectures.}

\smallskip\noindent
--------------------------------

 

\tableofcontents

\section{Introduction}

Let $V$ be a finite dimensional vector space over a given field $\K$, and consider a surjective endomorphism $\phi\colon V\to V$. Then $\phi$ has to be injective, and one way to show this is to compute dimensions:
$$ \dim_{\K}(V)=\dim_{\K}(\phi(V))+\dim_{\K}(\ker(\phi))=\dim_{\K}(V)+\dim_{\K}(\ker(\phi))\ \ \ \text{so that} \ \ \ \dim_{\K}(\ker(\phi))=0\,,$$
hence $\ker(\phi)=0$. A left module $M$ over a ring $R$ is said to be {\em Hopfian} if any of its surjective endomorphisms is an automorphism.
Thus, the above argument shows that finite dimensional vector spaces are Hopfian modules. Repeating these computations, but with composition length $\ell(-)$ of modules used instead of dimension, one can show that modules of finite length over any ring are Hopfian. 

\medskip
Let us try to abstract the above idea: fix a ring $R$ and consider a function $L\colon \lmod R\to \R_{\geq0}\cup\{\infty\}$, 
that associates to any left $R$-module $M$ the value $L(M)\in \R_{\geq0}\cup\{\infty\}$. We say that $L$ is 
\begin{enumerate}[\rm --]
\item {\em additive} if, given a short exact sequence 
 $0\to A\to B\to C\to 0$ in $\lmod R$,  $L(B)=L(A)+L(C)$;
\item  {\em faithful} provided $L(M)=0$ if and only if $M=0$. 
\end{enumerate}
If we have a faithful additive function on $\lmod R$, then we can conclude in the exact same way that all the modules in the class  $\Fin(L):=\{M\in\lmod R:L(M)<\infty\}$
are Hopfian.

\medskip
A rich source of Hopfian modules is the class of Noetherian modules. The usual proof that a Noetherian module $M$ is Hopfian is as follows:  any  surjective endomorphism $\phi\colon M\to M$ induces isomorphisms
$$\ker(\phi)\overset{\cong}{\leftarrow} \ker(\phi^2)/\ker(\phi)\overset{\cong}{\leftarrow} \ker(\phi^3)/\ker(\phi^2)\overset{\cong}{\leftarrow} \ldots$$
By the ascending chain condition, the chain $\ker(\phi)\subseteq \ker(\phi^2)\subseteq \ker(\phi^3)\subseteq \ldots$ stabilizes, so there exists $n\in\N$ such that $0\cong \ker(\phi^{n+1})/\ker(\phi^n)\cong \ker(\phi)$.  There is another way to show that Noetherian modules are Hopfian: using additive functions. In fact, in this generality, it is not possible to use just one faithful additive function, but we need to introduce an entire family of (not always faithful) additive functions that, when used together, allow us to conclude: 

\smallskip
The {first step} is to notice that, given a hereditary torsion class $\T\subseteq \lmod R$, there is an additive function 
$$\ell_\T\colon \lmod R\to \R_{\geq0}\cup\{\infty\}\,,\ \ \text{ defined as }\ \ \ \ell_{\T}(M)=\sup_{\sigma}\ell_\T(\sigma)\,,$$ where $\sigma$ ranges in all the finite sequences of the form $\sigma: 0=M_0\subseteq M_1\subseteq \ldots \subseteq M_n=M$ and $\ell_\T(\sigma):=\#\{i\in\{1,\dots,n\}:M_{i}/M_{i-1}\notin \T\}$. In other words, $\ell_\T(M)$ is the composition length of  $M$ in the Gabriel quotient $\lmod R/\T$ (see Section \ref{LF}). Notice that $\ell_\T(M)=0$ if and only if $M\in \T$.

Now that we have a procedure that takes as an input a hereditary torsion class and produces an additive function, we need a source for torsion classes. Thus, our {second step} is to recall that the {\em Gabriel filtration} of $\lmod R$ is a transfinite sequence of hereditary torsion classes $\C_\alpha\subseteq \lmod R$ such that:
\begin{enumerate}[\rm --]
\item $\C_0$ is the class containing just the $0$-modules;
\item if $\C_{\alpha}$ is given, $\C_{\alpha+1}$ is the smallest hereditary torsion class containing $\C_{\alpha}$ and all the modules $M$ such that, given $N\leq M$, exactly one between $N$ and $M/N$ is in $\C_\alpha$ (these are those $M$ that become simple in the Gabriel quotient $\lmod R/\C_\alpha$);
\item if $\lambda$ is a limit ordinal,  $\C_\lambda$ is the smallest hereditary torsion class containing all the $\C_\alpha$, for $\alpha<\lambda$.
\end{enumerate}
For any ordinal $\alpha$, we denote by $\ell_\alpha\colon \lmod R\to \R_{\geq0}\cup\{\infty\}$ the additive function induced by $\C_\alpha$ with the above procedure.

At this point, we can conclude with few observations: given a Noetherian left $R$-module $M$ and a surjective endomorphism $\phi\colon M\to M$ suppose, looking for a contradiction, that $\ker(\phi)\neq 0$. Using that $\ker(\phi)$ is Noetherian, we can find an ordinal $\alpha$ such that $\ker(\phi)\in\C_{\alpha+1}\setminus\C_\alpha$. Considering the $\alpha+1$-torsion part of $M$, $T_{\alpha+1}M:=\{x\in M:Rx\in \C_{\alpha+1}\}$, we obtain the following short exact sequence: 
$$0\to \ker(\phi)\to T_{\alpha+1}M\to T_{\alpha+1}M\to 0\,.$$
Furthermore, $\ell_{\alpha}(\ker(\phi))\leq \ell_{\alpha}(T_{\alpha+1}M)<\infty$ (where the finiteness of these quantities is a consequence of Noetherianity, see Lemma \ref{pre1}. A simple way to see this is that any object in the quotient category $\C_{\alpha+1}/\C_\alpha$ is semi-Artinian and objects that are both Noetherian and semi-Artinian do have finite length). Thus, by the additivity of $\ell_\alpha$, $\ell_\alpha(\ker(\phi))=0$, that is, $\ker(\alpha)\in \C_{\alpha}$, which is a contradiction. Thus $\ker(\phi)=0$ and $\phi$ is an automorphism as desired.

\medskip
To summarize, the above strategy consists in: finding a suitable filtration of $\lmod R$ by hereditary torsion classes, associate to any of these torsion classes an additive function, and finally select a suitable length function in our family to prove the injectivity of the given map. One of the main observation of this paper is that this same strategy can be adapted to find a partial solution of the following famous conjecture:

\medskip\noindent
{\bf Stable Finiteness Conjecture} {\rm (Kaplansky, 1972)}{\bf .} 
{\em Let $\K$ be a field and let $G$ be a group. Then, the group ring $\K[G]$ is stably finite. }
\bigskip

The above conjecture was verified by Kaplansky \cite{Kap} in case $\K=\mathbb C$, and similar arguments allow to solve the conjecture in the positive for any field of zero characteristic. Furthermore, it was proved by Ara, O'Meara and Perera \cite{Ara} that $\K[G]$ is stably finite whenever $G$ is residually amenable and $\K$ is an arbitrary division ring. Soon after, this last result was extended by Elek and Szab\'o \cite{Elek}, who proved that the  group ring $\K[G]$ is stably finite for $\K$ any division ring and $G$ a sofic group. Let us also remark that Ara et al. proved that any crossed product $\K\asterisk G$ (not just the group ring $\K[G]$, see Section \ref{crossed_subs}) is stably finite, provided $\K$ is a division ring and $G$ is an amenable group (see Section \ref{am_groups}).

\medskip
The relation between the above conjecture and Hopficity is clear: the stable finiteness of the group ring $\K[G]$ means exactly that all the free $\K[G]$-modules of finite rank $\K[G]^n$ are Hopfian. In this paper we extend, and provide a new proof for, the known positive solution of the above conjecture in  case  $G$ is an amenable group. Our main result in this direction will be the following:

\medskip\noindent
{\bf Theorem A.} {\em 
Let $R$ be a left Noetherian ring, let $G$ be a finitely generated amenable group and let $\RG$ be a fixed crossed product. Then, for any finitely generated left $R$-module $_RK\in\lmod R$, any $\RG$-submodule of $\RG\otimes_RK$ is Hopfian.\\ In particular, $\End_{\RG}(M)$ is stably finite for any submodule ${}_{\RG}M\leq \RG\otimes_RK$.
}
\bigskip 
 
Notice that the above theorem is an extremely strong version of the Stable Finiteness Conjecture that cannot be extended to sofic groups. In fact, the above statement fails for free non-Abelian groups (see Example \ref{libero} and Problem \ref{problem}). On the other hand, recently Li and Liang \cite{Li:2015aa} were able to adapt our strategy to show that $R[G]$ is stably finite whenever $R$ is left Noetherian and $G$ is sofic. It is still an open question if the same is true for more general crossed products.
 
 \medskip
Let us briefly describe our strategy to prove Theorem A. The idea is very close to the one we sketched above in proving the Hopficity of Noetherian modules using additive functions. Fix a crossed product $\RG$, with $G$ finitely generated amenable and $R$ Noetherian. With a little abuse of notation, for any ordinal $\alpha$ we denote by $\C_\alpha$ the hereditary torsion class of those $\RG$-modules ${}_{\RG}M$ whose underlying $R$-module ${}_RM$ is in $\C_\alpha$ (the $\alpha$-th layer of the Gabriel filtration of $\lmod R$, as described above). Now we need suitable additive functions of $\lmod {\RG}$, one for each $\alpha$. To do this, we need to ``extend" the additive functions $\ell_\alpha\colon \lmod R\to \R_{\geq0}\cup\{\infty\}$ from $\lmod R$ to $\lmod {\RG}$. This problem of extension of a given additive function is an interesting matter of dynamical nature, that deserves a brief discussion in its own (see Theorem B below). For the moment,  it is enough for us to know that this is the point where we use the amenability of $G$ to produce, for any ordinal $\alpha$, an additive function 
$L_\alpha\colon \lmod {\RG}\to \R_{\geq0}\cup\{\infty\}$
such that:
\begin{enumerate}[\rm --]
\item ``$L_\alpha$ extends $\ell_\alpha$", that is, $L_\alpha(\RG\otimes_{R}K)=\ell_\alpha(K)$ for any left $R$-module $K$ such that $\ell_\alpha(K)<\infty$;
\item ``$L_\alpha$ is as faithful as it can be", that is, $L_\alpha({}_{\RG}N)=0$ if and only if $\ell_\alpha({}_RN)=0$, for any $\RG$-submodule $N\leq \RG\otimes_{R}K$, with $K$ as above.
\end{enumerate}
Now it is easy to conclude: let $K$ be a finitely generated left $R$-module, consider an $\RG$-submodule $M\leq \RG\otimes_RK$, and let $\phi\colon M\to M$ be a surjective endomorphism. Suppose, looking for a contradiction, that $\ker(\phi)\neq 0$. Using the Noetherianity of $K$, it is possible to find an ordinal $\alpha$ such that ${}_R\ker(\phi)\in\C_{\alpha+1}\setminus \C_{\alpha} $. Then, there is a short exact sequence of $\RG$-modules
$$0\to \ker(\phi)\to T_{\alpha+1}M\to T_{\alpha+1}M\to 0\,.$$
Furthermore, since $T_{\alpha+1}(\RG\otimes_{R}K)\cong RG\otimes_{R}T_{\alpha+1}K$,
$$L_{\alpha}(\ker(\phi))\leq L_\alpha(T_{\alpha+1}M)\leq L_\alpha(\RG\otimes_{R}T_{\alpha+1}K)=\ell_\alpha(T_{\alpha+1}K)<\infty$$
where the equality uses the first property of $L_\alpha$ listed above. By the additivity of $L_\alpha$, we obtain that $L_\alpha(\ker(\phi))=0$ and so, by the second property listed above, $\ell_\alpha({}_R\ker(\phi))=0$, so $\ker(\phi)\in\C_{\alpha}$, which is a contradiction, so $\ker(\phi)=0$ and $\phi$ is an automorphism as desired.

\medskip
As we explained above, in order to apply the machinery of torsion classes and additive functions to Kaplansky's Stable Finiteness Conjecture we need to answer (at least partially), the following question:
\begin{quotation}\noindent
{\em Let $R$ be a ring, let $G$ be an amenable group and fix a crossed product $\RG$. Is it possible to find a map 
\begin{align*}\{\text{additive functions on $\lmod R$}\}&\longrightarrow \{\text{additive functions on $\lmod {\RG}$}\}\\
L&\longmapsto L_{\RG}\end{align*}
such that $L_{\RG}(N)=0$ implies $L(N)=0$ for any $N\leq {\RG}\otimes_RK$, and $L_{\RG}({\RG}\otimes_RK)=L(K)$, for any $K\in\Fin(L)$?}
\end{quotation}
In fact we will restrict our attention to a particular class of additive functions called {\em length functions}, where an additive function $L\colon \lmod R\to \R_{\geq0}\cup\{\infty\}$ is a length function provided 
$$L(M)=\sup\{L(F):\text{$F\leq M$ finitely generated}\}\,.$$
The composition length $\ell$, and more generally all the additive functions of the form $\ell_\alpha$ arising from the Gabriel filtration, are in fact length functions. Let us also remark that, even when $R$ is a field and $L$ is the dimension of vector spaces, Elek \cite{Elek_addition} showed that it is not always possible to extend $L$ to a length function of $\lmod {R[G]}$ if $G$ is not an amenable group, so it is essential to include amenability in our hypotheses. In fact, Elek \cite{Elek_dual} showed that  (again for $R$ a field and $L$ the dimension) this is always possible when $G$ is finitely generated amenable.

\medskip
Given a ring $R$  (no hypothesis is needed on the ring) and a finitely generated amenable group $G$, fix a crossed product $\RG$ and a length function $L\colon \lmod R\to \R_{\geq 0}\cup\{\infty\}$ which is compatible with $\RG$ (see Definition \ref{crossed_comp}). In particular, all the functions $\ell_\alpha$ arising from some layer in the Gabriel filtration satisfy these hypotheses. \\
Given a left $\RG$-module ${}_{\RG}M$, let $\Fin_L(M):=\{{}_RK\leq M:L(K)<\infty\}$. If $\Fin_L(M)$ contains all the finitely generated $R$-submodules of $M$ we say that $M$ is {\em locally $L$-finite} and we denote by $\lFin_L(\RG)$ the class of such modules. For example, $\lFin_L(\RG)$ contains all the left $\RG$-modules ${}_{\RG}M$ such that $L(M)<\infty$. Furthermore, given a left $R$-module $_RK$ such that $L(K)<\infty$, then the left $\RG$-module $\RG\otimes_RK$ is in $\lFin_L(\RG)$. In Section \ref{amenable_ent} we define the {\em algebraic $L$-entropy} as a function
$$\ent_L\colon\lFin_L(\RG) \to \R_{\geq 0}\cup\{\infty\}\,.$$ 
The main result of Section \ref{amenable_ent} is the following theorem, showing that the algebraic entropy provides a partial answer to above extension problem:

\medskip\noindent
{\bf Theorem B.}
{\em In the above hypotheses and notation, the invariant $\ent_L:\lFin_L(\RG)\to \R_{\geq 0}\cup\{\infty\}$ satisfies the following properties:
\begin{enumerate}[\rm (1)]
\item $\ent_L$ is a length function;
\item $\ent_L(\RG\otimes_{R}K)=L(K)$ for any $L$-finite left $R$-module $K$;
\item $\ent_L({}_{\RG}N)=0$ if and only if $L({}_RN)=0$, for any $\RG$-submodule $N\leq \RG\otimes_{R}K$.
\end{enumerate} }

\medskip
The proof of the above theorem will take almost all Section \ref{amenable_ent}, in particular, part (1) of the theorem will be proved in Sections \ref{upp_cont_subsec} and \ref{AT_subsec}, while parts (2) and (3) in Section \ref{shift_subsec}.\\

The final Section \ref{Kap_and_rel} is devoted to two applications of the algebraic entropy: the already mentioned Stable Finiteness Conjecture, and the Zero Divisors Conjecture, which states that, for a torsion-free group $G$ and a field $\K$, the group ring $\K[G]$ is a domain. Using the algebraic entropy we will give a new proof of the fact, known for group rings, that, provided $G$ is amenable, a given crossed product $\K\ast G$ is a domain if and only if it is an Ore domain. Furthermore, using an approach similar to that of Chung and Thom \cite{CT}, we will show that the Zero Divisors Conjecture is equivalent to a statement about the possible range of the algebraic entropy, extending their results from finite fields to arbitrary skew fields, and from group rings to crossed products (see Theorem \ref{dom=oredom}).

\bigskip\noindent
{\bf Acknowledgements.} It is a pleasure for me to thank Peter V\'amos for giving me a copy of his Ph.D. thesis and for useful discussions started in Padova in 2010. I am also sincerely grateful to my Ph.D. advisor Dolors Herbera for her encouragement and trust: she gave me time, freedom and several suggestions that helped me to work independently on this project. Finally, I would like to thank Pere Ara, Ferran Ced\'o, Hanfeng Li, Bingbing Liang and Nhan-Phu Chung for some useful discussions on preliminary versions of this paper.

\section{Length functions and torsion theories}\label{LF}

All along this section, unless otherwise stated, $\C$ denotes a Grothendieck category, that is, a cocomplete Abelian category, with exact directed colimits (i.e., $\C$ satisfies Grothendieck's axiom (Ab.5)) and with a generator. For unexplained terms and notations we refer to \cite{Ste}.

\medskip
Remember that any Grothendieck category is  well-powered (see \cite[Proposition 6, p.94]{Ste}) and that the set of all subobjects of a given object $M$ is a bounded and bicomplete lattice, which we denote by $\L(M)$; the minimum of $\L(M)$ is the $0$-object and the maximum is $M$ itself. For every family $\{N_i:i\in I\}$ of subobjects of $M$ we denote by $\sum_IN_i$ (resp., $\bigcap_{I}N_i$) the join (resp., meet) of the $N_i$ in $\L(M)$. If $\{N_i:i\in I\}$ is a directed system of subobjects, $\sum_IN_i$ is called the  {\em direct union} of the $N_i$ and it is sometimes denoted also by $\bigcup_IN_i$.
With this notation we can state the equivalent form of Grothendieck's axiom Ab.5 stating that, given an object $M$ in $\C$, a subobject $K$ of $M$ and a directed system  $\{N_i: i\in I\}$ of subobjects of $M$,  
$$\left(\sum_I N_{i}\right)\cap K=\sum_I(N_{i}\cap K)\, .$$
Let us recall the following fundamental properties of Grothendieck categories:
\begin{fact}{\rm \cite[Corollaries 4.3 and 4.4, pp.222--223]{Ste}}
Let $\C$ be a Grothendieck category, then
\begin{enumerate}[\rm (1)]
\item $\C$ is complete;
\item $\C$ has enough injective objects.
\end{enumerate}
\end{fact}
Given a family $\{M_i:i\in I\}$ of objects of $\C$, we denote by $\prod_I M_i$ and $\bigoplus_I M_i$ the product and the coproduct respectively. 
Given an object $M$ of $\C$, we denote by $E(M)$ the injective envelope of $M$. Given an endomorphism $\phi\colon M\to M$, a subobject $N$ of $M$ is {\em $\phi$-invariant} if $\phi(N)\leq N$.
Any subobject of a quotient (or, equivalently, a quotient of a subobject) of an object $M$ is generically called a {\em segment} of $M$.
A {\em series} of $M$ is a finite chain of subobjects 
\begin{equation}\label{filt}\sigma:\ \ 0=N_0\leq N_1\leq \dots\leq N_n=M\, .\end{equation}
The {\em factors} of $\sigma$ are the segments of $M$ of the form $N_{i+1}/N_i$ for some $i<n$. Two series of $M$ are {\em equivalent} if they have the isomorphic factors up to reordering. If all the factors of $\sigma$ are simple objects, then we say that $\sigma$ is a {\em composition series} for $M$. If $\sigma'\ \ 0=N'_0\leq N'_1\leq \dots\leq N'_m=M$ is another series of $M$, we say that $\sigma'$ is a refinement of $\sigma$ if $\{N_0,\dots,N_n\}\subseteq \{N_0',\dots,N_m'\}$.
The following fact is well-known.
\begin{fact}[Artin-Schreier's Refinement Theorem]\label{A-S}
Let $\C$ be a Grothendieck category and let $M$ be an object of $\C$. If $\sigma_1$ and $\sigma_2$ are two series of $M$, then they admit equivalent refinements.
\end{fact}

\subsection{Length functions}
In any category $\C$ it is possible to define (real-valued) invariants to measure various finiteness properties of the objects. In general, we call {\em invariant} of $\C$, any map $i\colon \C\to \R_{\geq 0}\cup\{\infty\}$ such that $i(X)=i(X')$ whenever $X$ and $X'$ are isomorphic objects in $\C$. 

\medskip
If we make some stronger assumption on the structure of the category $\C$, we can refine our definition of invariant in order to obtain a more treatable notion. Indeed, suppose that $\C$ is an Abelian category (or, more generally, an exact category). The information that one usually wants to encode is about homological properties. In this setting it seems natural to ask that, given a short exact sequence
\begin{equation}\label{ses}0\to X_1\to X_2\to X_3\to 0\end{equation}
in $\C$, we have $i(X_2)=i(X_1)+i(X_3)$. In this case, we say that $i$ is {\em additive} on the sequence \eqref{ses}. If $i$ is additive on all the short exact sequences of $\C$ and $i(0)=0$, then we say that $i$ is an {\em additive invariant} (or {\em additive function}). \\
In the following lemma we collect some useful properties of additive functions.

\begin{lemma}\label{basic_add}
Let $\C$ be an Abelian category and let $i\colon \C\to \R_{\geq 0}\cup\{\infty\}$ be an additive function. Then
\begin{enumerate}[\rm (1)]
\item $i(X)\geq i(Y)$ for every segment $Y$ of $X\in \C$;
\item $i(X_1+X_2)+i(X_1\cap X_2)=i(X_1)+i(X_2)$ for every pair of subobjects $X_1,X_2$ of $X\in\C$;
\item $\sum_{j \text{odd}}i(X_j)=\sum_{j \text{even}} i(X_j)$ for every exact sequence $0\to X_1\to X_2\to \dots\to X_n\to 0$ in $\C$.
\end{enumerate}
\end{lemma}

A natural assumption in the context of Grothendieck categories is the upper continuity of invariants, given an object $X\in \C$ and a directed set $\mathcal S=\{X_\alpha:\alpha\in \Lambda\}$ of subobjects of $X$ such that $\sum_\Lambda X_\alpha=X$, we say that $i$ is {\em continuous} on $\mathcal S$ if
\begin{equation}\label{upper_cont}
i(X)=\sup\{i(X_\alpha):\alpha\in \Lambda\}\, .
\end{equation} 
If $i$ is continuous on all the direct unions as above, we say that $i$ is {\em upper continuous}. Upper continuity can be defined in arbitrary Abelian categories even if it seems more meaningful when direct limits exist and are exact.
\begin{definition}[\rm \cite{NR}]
Let $\C$ be an Abelian category. An additive and upper continuous invariant $i\colon \C \to \R_{\geq 0}\cup\{\infty\}$ is said to be a {\em length function}. 
\end{definition}

In what follows we generally denote  length functions by the symbol $L$.  We remark that the usual definition of length function is given in module categories, which are in particular locally finitely generated Grothendieck categories. In this special setting, the usual definition of upper continuity is different (see part (3) of the following proposition). We now show that we are not defining a new notion of upper continuity but just generalizing this concept to arbitrary Grothendieck categories (similar observations, with analogous proofs, were already present in \cite{V2} for module categories).

\smallskip
Recall that, given an object $M$ of a Grothendieck category $\C$ and an ordinal $\kappa$, a set $\{M_\alpha:\alpha<\kappa\}$ of subobjectsis of $M$ is a {\em continuous chain} if
\begin{enumerate}[\rm (1)]
\item $M_\alpha\leq M_{\beta}$, provided $\alpha,\beta <\kappa$ and $\alpha\leq \beta$;
\item $\bigcup_{\alpha<\lambda}M_\alpha=M_\lambda$ for any limit ordinal $\lambda<\kappa$.
\end{enumerate}

\begin{proposition}\label{lf}
Let $\C$ be a Grothendieck category and $L\colon \C\to \R_{\geq 0}\cup\{\infty\}$ be an additive function. Consider the following statements:
\begin{enumerate}[\rm (1)]
\item $L$ is a length function;
\item given an object $M\in \C$, an ordinal $\kappa$ and a continuous chain $\{M_{\alpha}:\alpha<\kappa \}$ of subobjects of $M$ such that $M=\bigcup_{\alpha<\kappa}M_\alpha$, we have that $L(M)=\sup \{L(M_\alpha):\alpha<\kappa\} ;$
\item for every object $M\in \C$ we have that $L(M)=\sup \{L(F):\text{$F$ finitely generated subobject of $M$}\}.$
\end{enumerate}
Then (1)$\Leftrightarrow$(2) and (2)$\Leftarrow$(3). If $\C$ is locally finitely generated, then the above statements are all equivalent.
\end{proposition}

In the last part of this subsection we show that, under suitable conditions, to prove the additivity of an invariant it is enough to show that it is additive on short exact sequences of finitely generated objects (see Proposition \ref{add_on_fg}). This will be extremely useful in proving that the algebraic entropy is additive. 

\begin{lemma}
Let $I$ be a set, let $\{a_i:i\in I\}$ and $\{b_i:i\in I\}$ be bounded subsets of $\R_{\geq0}$, and let $a:=\sup_ia_i$ and $b:=\inf_i b_i$. If $a_i\leq a$, $b\leq b_i$ and $a_i+b_i=c$ for all $i\in I$, then $a+b=c$.
\end{lemma}
\begin{proof}
For any $\varepsilon>0$ there exists $i\in I$ such that $|a-a_i|<\varepsilon$ and $|b-b_i|<\varepsilon$. Thus,
\begin{equation*}c-\varepsilon=a_i+(b_i-\varepsilon) <a+b<(a_i+\varepsilon)+b_i=c+\varepsilon\,.\qedhere\end{equation*}
\end{proof}

\begin{proposition}\label{add_on_fg}
Let $R$ be a ring and let $\C\subseteq \lmod R$ be a subclass closed under subobjects and quotients. An upper continuous invariant $L\colon \C\to \R_{\geq0}\cup\{\infty\}$ such that $L(F)<\infty$ for any finitely generated $F\in \C$ is additive if and only if the following properties hold true:
\begin{enumerate}[\rm (1)]
\item $L(K_2)=L(K_1)+L(K_3)$ for any short exact sequence $0\to K_1\to K_2\to K_3\to 0$ of finitely generated modules in $\C$;
\item given a finitely generated $K\in\C$ and a directed system of submodules $\{H_i\}_i$,  $L(K/H)=\inf_i L(K/H_i)$, where $H:=\bigcup_iH_i$.
\end{enumerate}
\end{proposition}
\begin{proof}
Consider a short exact sequence $0\to N\to M\to M/N\to 0$ in $\C$. Then, $L(M)=\sup_{F}L(F)$, with $F$ ranging in the finitely generated submodules of $M$. Furthermore, for any such $F$ we have that $L(N\cap F)=\sup_{G}L(G)$, $L(F/(N\cap F))=\inf_GL(F/G)$ and $L(G)+L(F/G)=L(F)$, for any $G\leq N\cap F$ finitely generated (where for these three equalities we used upper continuity, hypothesis (2) and hypothesis (1), respectively). By the above lemma, this implies that $L(F)=L(N\cap F)+L(F/(N\cap F))$. Thus,
\begin{align*}
L(M)&=\sup_{F\leq M\text{ f.g. }} L(F)\\
&=\sup_{F\leq M\text{ f.g. }} L(N\cap F)+L(F/(N\cap F))\\
&=\sup_{F\leq M\text{ f.g. }}  L(N\cap F)+ \sup_{F\leq M\text{ f.g. }} L((F+N)/N)\\
&=L(N)+L(M/N)
\end{align*}
On the other hand, if $L$ is additive, then (1) is clear. Furthermore, given a finitely generated $K\in\C$ and a directed system of submodules $\{H_i\}_i$, then
\begin{equation*}L(K/H)=L(K)-L(H)=L(K)-\sup_{i}L(H_i)=\inf_i(L(K)-L(H_i))=\inf_iL(K/H_i)\,.\qedhere\end{equation*}
\end{proof}

\subsection{Gabriel localization and Gabriel dimension}\label{gab_cat}

Let $\C$ be a Grothendieck category and let $\A\subseteq \C$ be a subclass. Recall that $\A$ is a {\em semi-closed} if, given a short exact sequence
$$0\to A\to B\to C\to 0\,,$$
$B\in \A$ implies that both $A$ and $C$ belong to $\A$. A {\em Serre class} is a semi-closed class that is closed under extensions. Furthermore, $\A$ is a {\em hereditary torsion class} (or {\em localizing class}) if it is Serre and it is closed under taking arbitrary direct sums. On the other hand,  $\A$ is a {\em hereditary torsion free class} provided it is closed under taking sub-objects, extensions, products and injective envelopes. 

\begin{definition}\label{def_tt}
A {\em hereditary torsion theory} $\tau$ in $\C$ is a pair of classes $(\T, \F)$ such that 
\begin{enumerate}[\rm --]
\item the class $\T$ of {\em $\tau$-torsion} objects is a hereditary torsion class;
\item the class $\F$ of {\em $\tau$-torsion free} objects is a hereditary torsion free class;
\item $(\T)^{\perp}=\F$ and $^{\perp}(\F)=\T$.
\end{enumerate}
In this paper the symbols $\tau$, $\T$ and $\F$ are always used to denote a torsion theory, a torsion class and a torsion free class respectively. Since all the torsion theories in the sequel are hereditary, we just say ``torsion theory'', ``torsion class" and ``torsion free class" to mean respectively ``hereditary torsion theory'', ``hereditary torsion class'' and ``hereditary torsion free class''. 
\end{definition}

Given a torsion theory $\tau=(\T,\F)$, the inclusion $\T\to \C$ has a right adjoint $\tor_\tau\colon \C\to \T$, which is a subfunctor of the identity, called the {\em torsion functor}. For any object $X\in \C$ there is a short exact sequence of the form
$$0\to \tor_\tau X\to X\to X/\tor_\tau X\to 0$$
such that $\tor_\tau X\in \T$ and $X/\tor_\tau X\in \F$. 

\medskip
Let us also recall the close relation between torsion theories and localization of Abelian categories. We start with the following definition:

\begin{definition}
A {\em localization} of $\C$ is a pair of adjoint functors $\Q:\C\rightleftarrows \D :\S$, where $\D$ is an Abelian category, $\S$ is fully faithful,  and $\Q$ is exact and essentially surjective. In this situation, $\D$ is called a {\em quotient category}, $\Q$ is a {\em quotient functor} and $\S$ is a {\em section functor}. The composition $\loc=\S\circ\Q\colon \C\to \C$ is called the {\em localization functor}.
\end{definition}

One can encounter slightly different definitions of localization in other contexts, see for example \cite{Localization}.
We can now explain the connection between localizations and torsion theories. Indeed, starting with a localization $\Q:\C\rightleftarrows \D:\S$ and letting $\loc=\S\circ\Q$, 
$$\ker(\loc)=\{X\in \C : \loc(X)=0\}=\{X\in \C : \Q(X)=0\}=\ker (\Q)$$ 
is a torsion class (use the exactness of $\Q$ and the fact that it is a left adjoint). Hence, the localization $(\Q,\S)$  induces a torsion theory $(\ker(\Q),\ker(\Q)^{\perp})$.
On the other hand, one can construct a localization out of a torsion theory. 
\begin{definition}
Let $\C$ be a Grothendieck category and $\tau=(\T,\F)$ a torsion theory. An object $X\in \C$ is {\em $\tau$-local} if $X\in\F$ and $E(X)/X\in \F$. 
The full subcategory of $\C$ of all the $\tau$-local objects is denoted by $\C/\T$. 
\end{definition}

The definition of the localization induced by a torsion theory depends on the following lemma.

\begin{lemma}{\rm \cite{Gabriel}}
Let $\C$ be a Grothendieck category and let $\tau=(\T,\F)$ be a torsion theory. Then, the canonical inclusion $\C/\T\to \C$ has a left adjoint functor which is exact.
\end{lemma}

\begin{definition}
Let $\C$ be a Grothendieck category and let $\tau=(\T,\F)$ be a torsion theory. The inclusion  \mbox{$\S_\tau:\C/\T\to \C$} is called {\em $\tau$-section functor}, while its left adjoint functor $\Q_\tau:\C\to \C/\T$ is called {\em $\tau$-quotient functor}. The composition $\loc_\tau=\S_\tau\Q_\tau$ is called {\em $\tau$-localization functor}.\end{definition}

Let us now recall that the {\em Gabriel filtration} of $\C$ is a transfinite chain 
$\{0\}=\C_{-1}\subseteq \C_0\subseteq \dots \subseteq \C_\alpha\subseteq \dots$
of torsion classes defined as follows:
\begin{enumerate}[--]
\item $\C_{-1}=\{0\}$;
\item suppose that $\alpha$ is an ordinal for which $\C_{\alpha}$ has already been defined. An object $C\in \C$ is said to be {\em $\alpha$-cocritical} if $C$ is $\tau_\alpha$-torsion free and every proper quotient of $C$ is $\tau_\alpha$-torsion, where $\tau_{\alpha}$ is the unique torsion theory whose torsion class is $\C_{{\alpha}}$. In the sequel, we will just say $\alpha$-torsion (resp., torsion free) instead of $\tau_\alpha$-torsion (resp., torsion free). 
We let $\C_{\alpha+1}$ be the smallest hereditary torsion class containing $\C_{{\alpha}}$ and all the ${\alpha}$-cocritical objects;
\item if $\lambda$ is a limit ordinal, we let $\C_\lambda$ be the smallest hereditary torsion class containing $\bigcup_{\alpha<\lambda}\C_\alpha$.
\end{enumerate}
For any ordinal $\alpha$, we let $\tor_\alpha\colon \C\to \C_{\alpha}$ and $\Q_\alpha\colon \C\to\C/\C_{\alpha}$ be respectively the torsion and the quotient functors. Abusing notation, we use the same symbols for the functors $\tor_\alpha\colon \C_{\alpha+1}\to \C_{\alpha}$ and $\Q_\alpha\colon \C_{\alpha+1}\to\C_{\alpha+1}/\C_{\alpha}$, induced by restriction.

\begin{remark}
By definition, a Grothendieck category $\C$ has a generator $G$. Furthermore, $G$ has just a set (as opposed to a proper class) of subobject, equivalently, it has just a set of quotients. One can show that $\C_{\alpha+1}$ is the smallest  torsion class containing $\C_\alpha$ and the $\alpha$-cocritical quotients of $G$. As a consequences we obtain that there is an ordinal $\kappa$ such that $\C_\alpha=\C_\kappa$ for all $\alpha\geq \kappa$ (just take $\kappa=\sup\{\alpha:\text{there are $\alpha$-cocritical quotients of $G$}\}$).
\end{remark}

Consider the union $\bar \C=\bigcup_{\alpha}\C_\alpha$ of the Gabriel filtration (this makes sense by the above remark). An object belonging to $\bar \C$ is said to be an object with Gabriel dimension. The Gabriel dimension of $M$ is the minimal ordinal $\delta$ such that $M\in\C_\delta$, in symbols $\Gdim(M)=\delta$. 
In general, it may happen that $\C\neq \bar \C$; if $\C=\bar \C$, we say that $\C$ is a {\em Gabriel category} with {\em Gabriel dimension} $\Gdim (\C)=\kappa$, where $\kappa$ is the smallest ordinal such that $\C_\kappa=\C$. 

\medskip
In the following example we specialize some of the above notions to categories of modules. Even if the example is stated for modules, part (1) holds for any Grothendieck category, taking a generator in place of $_{R}R$. Part (2) can be generalized to any Grothendieck category with a set of Noetherian generators. 
\begin{example}
Let $R$ be a ring and let $\C=\lmod R$ be the category of left $R$-modules. 
\begin{enumerate}[\rm (1)]
\item If $\C$ has Gabriel dimension $\Gdim(\C)=\kappa$, then $\Gdim({}_RR)=\kappa$. In fact, the inequality $\Gdim(\C)\geq \Gdim({}_RR)$ is trivial. For the converse, just notice that if $\alpha$ is an ordinal for which ${}_RR\in \C_\alpha$, then $\C_\alpha=\C$. This is because $\C_\alpha$ is closed under arbitrary direct sums and quotients, so it contains any quotient of a free module, that is, any module.
\item If $R$ is left Noetherian then $\C$ is a Gabriel category. In fact, given an ordinal $\alpha$, $\C_{\alpha}=\C_{\alpha+1}$ implies that ${}_RR\in \C_{\alpha}$. To show this notice that, if $\alpha$ is an ordinal for which ${}_RR\notin\C_\alpha$, the set $\mathcal I=\{{}_RI\leq R: {}_R(R/I)\notin \C_\alpha\}$ is non-empty and, by Noetherianity, $\mathcal I$ has a maximal element $\bar I$. Then, ${}_R(R/\bar I)\notin \C_\alpha$ but any of its proper quotients is in $\C_\alpha$. So, ${}_R(R/\bar I)$ is $\alpha$-cocritic and it belongs to $\C_{\alpha+1}$, which therefore properly contains $\C_\alpha$.
\end{enumerate}
\end{example}

In what follows we collect some well-known properties of Gabriel dimension (for a proof see \cite{GR}). Recall that a Grothendieck category $\C$ is said to be {\em semi-Artinian} if $\Gdim(\C)=0$.

\begin{lemma}\label{pre1}
Given a Gabriel category $\C$, the following statements hold true:
\begin{enumerate}[\rm (1)]
\item $\Gdim(\C)=\sup\{G.\dim(M): M\in\C\}$;
\item if $N\leq M\in \C$, then $\Gdim(M)=\max\{\Gdim(N),\Gdim(M/N)\}$;
\item let $\alpha$ be an ordinal $< \Gdim(\C)$ and $M \in \C$, then $M\in \C_{\alpha+1}$ if and only if there exists an ordinal $\sigma$ and a continuous chain 
$0=N_0\leq N_1\leq \dots\leq N_\sigma=M$, such that $N_{i+1}/N_i$ is either $\alpha$-cocritical or $\alpha$-torsion for every $i< \sigma$;
\item let $\lambda$ be a limit ordinal $\leq \Gdim(\C)$ and $M \in \C$, then $M\in \C_{\lambda}$ if and only if $M=\bigcup_{\alpha<\lambda}\tor_\alpha(M)$;
\item $\C_{\alpha+1}/\C_\alpha$ is semi-Artinian for all $\alpha<\Gdim(\C)$;
\item let  $N\in\C$ be a Noetherian object, then $\Gdim(N)$ is a successor ordinal. Furthermore, there exists a finite series $0=N_0\leq N_1\leq\dots\leq N_k=N$ such that $N_{i}/N_{i-1}$ is cocritical for all $i=1,\dots,k$.
\end{enumerate}
\end{lemma}
%
%
%
%
%
%
%

\subsection{Length functions compatible with self-equivalences}\label{Sec_compatibility}\label{examples}

All along this section, let $\mathfrak C$ be a Grothendieck category, let $L\colon \C\to\R_{\geq0}\cup\{\infty\}$ be a length function and let $F\colon\C\to \C$ be a self-equivalence, that is, 
\begin{enumerate}[\rm (Eq. 1)]
\item $F$ is {\em essentially surjective}, i.e., for all $X\in \C$, there exists $Y\in \C$ such that $F(X)\cong Y$;
\item $F$ is {\em fully faithful}, i.e., for all $X,Y\in \C$, the natural morphism $\Hom_\C(X,Y)\to \Hom_\C(F(X),F(Y))$ is an isomorphism.
\end{enumerate}
A consequence of the definition is that $F$ preserves any structure defined by universal properties, in particular, $F$ commutes with direct and inverse limits and it preserves exactness of sequences. Furthermore, $F$ commutes with injective envelopes and it preserves lattices of subobjects. 
It is easily seen that 
\begin{equation}\label{def_L_F}
L_F:\C\to \R_{\geq0}\cup\{\infty\}\ \ \text{ such that }\ \  L_F(M)=L(F(M))\,,
\end{equation} for all $M\in\C$ is a length function. In what follows we are going to study to what extent $L_F$ can differ from $L$.
The following example shows that $L$ and $L_F$ may be very different.

\begin{example}
Consider a field $K$ and consider the category $\lmod {K\times K}\cong \lmod K\times \lmod K$. This category is semi-Artinian and it has a self-equivalence $F\colon\lmod {K\times K}\to \lmod {K\times K}$ such that $(M,N)\mapsto (N,M)$ and $(\phi,\psi)\mapsto (\psi,\phi)$. If we take $L$ to be the length function such that $L((M,N))=\dim_K(M)$, then clearly $L_F((M,0))=0\neq \dim_K(M)=L((M,0))$, provided $M\neq 0$. 
\end{example}

\begin{definition}
Given a Grothendieck category $\C$ and a self-equivalence $F\colon \C\to \C$, we say that a length function $L\colon\C\to \R_{\geq0}\cup\{\infty\}$ is {\em compatible} with $F$ provided  $L_F(M)=L(F(M))=L(M)$ for all $M\in\C$.
\end{definition}

Our motivation for studying compatibility of  length functions with  self-equivalences is the following: given a ring $R$  and  a ring automorphism $\phi\colon R\to R$, we define a self-equivalence 
\begin{equation}\label{eq_equiv_auto}F_\phi\colon \lmod R\to \lmod R\,,\end{equation}
which is the scalar restriction along $\phi$, that is: given a left $R$-module ${}_{R}M$ such that $R$ acts on $M$ via a ring homomorphism $\lambda\colon R\to \End_\Z(M)$,  $F_\phi(M)$ is isomorphic to $M$ as an Abelian group, while $R$ acts on $F_\phi(M)$ via the ring homomorphism $\lambda\circ\phi\colon R\to \End_\Z(F_\phi(M))$. We are interested in finding length functions $L$ such that $L(F_\phi(M))=L(M)$. 

\subsection{Examples of length functions}

\begin{example}\label{rk}
Let $D$ be a left Ore domain, let $\Sigma=D\setminus \{0\}$ and let $Q=\Sigma^{-1}D$ be the ring of left fractions of $D$. For any $M\in\lmod D$, the torsion free rank $\rk(M)$ is defined to be the dimension of the left module $\Sigma^{-1}M=Q\otimes_{D}M$ over the skew field $Q$. The torsion free rank $\rk\colon\lmod D\to \R_{\geq 0}\cup\{\infty\}$, 
where $\rk(M)=\infty$ whenever the rank of $M$ is not finite, is a length function. In particular, the dimension of vector spaces over a field is a length function. 
\end{example}

\begin{example}
The logarithm of the cardinality $\log|-|\colon\lmod \Z\to \R_{\geq 0}\cup\{\infty\}$, 
where $\log|G|=\infty$ whenever $G$ is not finite, is a  length function.
\end{example}

In what follows we define the length of a lattice. This can be applied to the lattice of subobjects of a given object in a Grothendieck category obtaining the so-called composition length, which is the invariant inspiring the abstract notion of length function.

\begin{definition}\label{lattice-length}
Let $(\L,\leq,0,1 )$ be a bounded lattice. A sequence of elements $\sigma:\, 0=x_0\lneq x_1\lneq \dots\lneq x_n=1$
is said to be a {\em series} of $\L$. Furthermore, the number of strict inequalities in a series $\sigma$, is called the {\em length} of $\sigma$, in particular the above series  has length $n$; we write it $\ell(\sigma)=n$. The {\em length} of $\L$ is 
$$\ell(\L)=\sup\{\ell(\sigma):\text{ $\sigma$ a series of $\L$}\}\,.$$
\end{definition}

\begin{example}\label{comp_length}
Let $\C$ be a Grothendieck category. Notice that a series of $\L(M)$ is just  a series of $M$. We can define a function $\ell\colon\C \to \R_{\geq 0}\cup\{\infty\}$, $\ell(M)=\ell(\L(M))$. 
Clearly, $\ell(M)=0$ if and only if $M=0$, $\ell(M)=1$ if and only if $M$ is a simple object and $\ell(M)<\infty$ if and only if $M$ is both Noetherian and Artinian. It is a standard result  that $\ell$ is a length function (see for example \cite{Fa} for a proof in module categories).
\end{example}

In the following proposition we introduce a family of length functions that will be extremely useful for our applications of algebraic entropy:

\begin{proposition}
Let $\C$ be a Grothendieck category, let $\alpha$ be an ordinal, and let $\C_\alpha$ be the $\alpha$-th layer in the Gabriel filtration of $\C$. Define an invariant
$$\ell_\alpha\colon \C\to \R_{\geq0}\cup\{\infty\}\ \ \ \text{such that }\ \ \ \ell_\alpha(M):=\ell(\L(\Q_\alpha M))$$
 is the length of the lattice of subobjects of $\Q_\alpha M$ in $\C/\C_{\alpha}$. Then, $\ell_\alpha$ is a length function, $\ker(\ell_\alpha)=\C_\alpha$, and $\ell_\alpha$ is compatible with any self-equivalence of $\C$. 
\end{proposition}
\begin{proof}
The fact that $\ell_\alpha$ is a length function follows since $\Q_\alpha$ is an exact left adjoint (so it commutes with direct limits and it sends short exact sequences to short exact sequences). The fact that $\ker (\ell_\alpha)=\C_\alpha$ follows since $\Q_\alpha(X)=0$ if and only if $X\in \C_\alpha$, and the unique  object in $\C/\C_\alpha$ whose composition length is zero is the $0$-object. The fact that $\ell_\alpha$  is compatible with any self-equivalence of $\C$ follows directly from the observation that, given $X\in \C$, the value $\ell_\alpha(X)$ just depends on the lattice of subobjects $\L(X)$, and any self-equivalence preserves lattices of subobjects. To show this one needs just to use the lattice theoretic description of Gabriel dimension (for this see \cite{Dim_ring_th}), and the lattice theoretic description of composition length given in Definition \ref{lattice-length}.
\end{proof}

\section{Amenable groups and crossed products}\label{infinite_groups}

\subsection{Amenable groups}\label{am_groups}
Amenable groups were defined by John von Neumann in 1929 as groups admitting a left-invariant mean. We adopt here an equivalent definition of amenability given by F\o lner \cite{Folner}. Indeed, consider two subsets $A,$ $C\subseteq G$, then
\begin{enumerate}[\rm --]
\item the {\em $C$-interior} of $A$ is $\In_C(A)=\{x\in G:xC\subseteq A\}$;
\item the {\em $C$-exterior} of  $A$ is $\Out_C(A)=\{x\in G:xC\cap A\neq \emptyset\}$;
\item the {\em $C$-boundary} of $A$ is $\partial_C(A)=\Out_C(A)\setminus \In_C(A)$.
\end{enumerate}
If $e\in C$, one can imagine the above notions as in the following picture
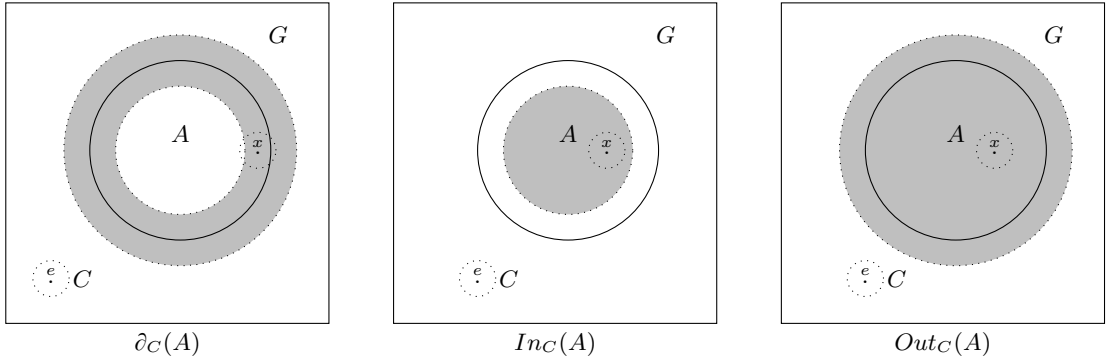
\begin{center}
\begin{tikzpicture}[scale = 0.85,fill=lightgray]
\scope
\fill (0.2,0.2) [lightgray]  circle (1.8);
\fill (0.2,0.2)  [white] circle (1);
\endscope
\scope
\fill (6.2,0.2) [lightgray]  circle (1);
\endscope
\scope
\fill (12.2,0.2) [lightgray]  circle (1.8);
\endscope
\draw (0.2,0.2) [dotted] circle (1.8) (0.2,0.2)  node [text=black,above] {$A$};
\draw (1.4,0.2) [dotted] circle (0.28) (1.4,0) node [text=black,above] {$\overset{x}{.}$};
\draw (-1.8,-1.8) [dotted] circle (0.28) (-1.8,-2) node [text=black,above] {$\overset{e}{.}$}(-1.6,-1.8) node [text=black,right] {$C$};
\draw     (0.2,0.2)  circle (1.4) (0.2,2)  node [text=black,above] {};
\draw     (0.2,0.2)  [dotted] circle (1) (0.2,2)  node [text=black,above] {};
\draw   (-2.5,-2.5) rectangle (2.5,2.5) (2,2) node [text=black,left] {$G$} (0,-2.5) node [text=black,below] {$\partial_C(A)$};
\draw (6.8,0.2) [dotted] circle (0.28) (6.8,0) node [text=black,above] {$\overset{x}{.}$};
\draw (4.8,-1.8) [dotted] circle (0.28) (4.8,-2) node [text=black,above] {$\overset{e}{.}$}(5,-1.8) node [text=black,right] {$C$};
\draw     (6.2,0.2)  circle (1.4) (6.2,0.2)  node [text=black,above]  {$A$};
\draw     (6.2,0.2)  [dotted] circle (1) (6.2,2)  node [text=black,above] {};
\draw   (3.5,-2.5) rectangle (8.5,2.5) (8,2) node [text=black,left] {$G$}(6,-2.5)node [text=black,below] {$\In_C(A)$};
\draw (12.2,0.2) [dotted] circle (1.8) (12.2,0.2)  node [text=black,above] {$A$};
\draw (12.8,0.2) [dotted] circle (0.28) (12.8,0) node [text=black,above] {$\overset{x}{.}$};
\draw (10.8,-1.8) [dotted] circle (0.28) (10.8,-2) node [text=black,above] {$\overset{e}{.}$}(11,-1.8) node [text=black,right] {$C$};
\draw     (12.2,0.2)  circle (1.4) (12.2,2)  node [text=black,above] {};
\draw   (9.5,-2.5) rectangle (14.5,2.5) (14,2) node [text=black,left] {$G$}(12,-2.5)node [text=black,below] {$\Out_C(A)$};
\end{tikzpicture}\end{center}
\begin{definition} 
A group $G$ is {\em amenable} if and only if there exists a directed set $(I,\leq)$ and a net $\{F_i:i\in I\}$ of finite subsets of $G$ such that, for any finite subset $C$ of $G$, 
\begin{equation}\label{lim_condition}\lim_{I}\frac{|\partial_C(F_i)|}{|F_i|}=0\, .\end{equation}
Any such net is called a {\em F\o lner net}. 
\end{definition}
If $G$ is countable, then one can take $I=\N$ and just speak about {\em F\o lner sequences}. Furthermore, given a countable group $G$ and a F\o lner sequence $\s=\{F_n:n\in\N\}$, we say that $\s$ is a {\em F\o lner exhaustion} if 
\begin{enumerate}[\rm (1)]
\item $e\in F_0$ and $F_n\subseteq F_{n+1}$ for all $n\in\N$;
\item $\bigcup_{n\in\N}F_n=G$.
\end{enumerate}
It can be proved  that a finitely generated group is amenable if and only if it admits a F\o lner exhaustion (see for example \cite[Lemma 5.3]{Pete}).

\begin{example}
\begin{enumerate}[\rm (1)]
\item Any finite group and any finitely generated Abelian group is amenable;
\item it is known that the class of amenable groups is closed under the operations of taking subgroups, taking factors over normal subgroups, taking extensions and taking increasing unions. We obtain that a group $G$ is amenable if and only if all its finitely generated subgroups are amenable, in particular, arbitrary Abelian groups and locally finite groups are amenable;
\item consider a finitely generated group $G$ and let $S=S^{-1}$ be a finite set of generators  containing the unit of $G$. For all $n\in\N_+$, let $B_n(S)=\{s_1\cdot\ldots\cdot s_n:s_i\in S\}$ and define a function:
$$f_S:\N_+\to \N_+ \ \ \text{such that}\ \ f_S(n)=|B_n(S)|\,.$$
The group $G$ is said to be of sub-exponential growth if the growth of $f_S$ is sub-exponential (it can be shown that this notion does not depend on the choice of the generating set $S$). If $G$ is of sub-exponential growth, then $\{B_n(S):n\in\N_+\}$ is a F\o lner exhaustion for $G$ (see for example \cite[Section 6.11]{Ceccherini_libro});
\item non-commutative free groups are not amenable. So, any group which contains a free subgroup of rank $2$ is not amenable. There exist amenable groups with exponential growth. 
\end{enumerate}
\end{example}

\subsection{Non-negative real functions on finite subsets of an amenable group}\label{quasi_tiling}
In their seminal paper \cite{OW}, Ornstein and Weiss introduced a notion of entropy for actions of amenable groups on metric spaces. Using the theory of quasi-tiles, they were able to prove that the $\limsup$ defining their entropy is a true limit. In what follows, we recall some of these deep results as they can be applied with just minor changes to our  algebraic setting. The following terminology and results are due to Ornstein and Weiss \cite{OW} (see also  \cite{HYG} and \cite{WZ}). 

\medskip
Denote by $\F(G)$ the family of all finite subsets of $G$. Let $A_1,\dots, A_k \in \F(G)$ and $\varepsilon \in (0, 1)$. The family $\{A_1,\dots, A_k\}$ is {\em $\varepsilon$-disjoint} if there are $A'_1,\dots, A'_k \in \F(G)$ such that
\begin{enumerate}[\rm (1)]
\item $A'_i \subseteq A_i$ and $|A'_i|/|A_i| > 1 - \varepsilon$ for $i = 1,\dots , k$;
\item $A'_i \cap A'_j=\emptyset$ if $1 \leq i \neq j \leq k$.
\end{enumerate}
Given $\alpha \in (0, 1]$ and $A\in \F(G)$,  $\{A_1,\dots , A_k\}$ is an {\em $\alpha$-cover} of $A$ if 
$$\frac{\left|A \cap ( \bigcup_{i=1}^k A_i)\right| }{|A|} \geq \alpha\,.$$ 
Finally,  $\{A_1, \dots , A_k\}$ {\em $\varepsilon$-quasi-tiles} $A$ if there exist $C_1, \dots , C_k\in\F(G)$ such that
\begin{enumerate}[\rm (1)]
\item  $C_iA_i \subseteq A$, for all $i = 1,\dots, k$, and $\{cA_i : c \in C_i\}$ forms an $\varepsilon$-disjoint family;
\item  $C_iA_i \cap C_jA_j = \emptyset$, if $1 \leq i \neq j \leq k$;
\item  $\{C_iA_i: i = 1,\dots , k\}$ forms a ($1-\varepsilon$)-cover of $A$.
\end{enumerate}
The subsets $C_1, \dots , C_k$ are called {\em tiling centers}. It is a deep result, due to Ornstein and Weiss, that whenever $G$ is an amenable group and $\{F_n\}_{n\in\N}$ is a F\o lner exhaustion, for any (small enough) $\varepsilon>0$, one can find a nice family of subsets of $G$ that $\varepsilon$-quasi-tiles $F_{n}$ for all (big enough) $n\in\N$. More precisely:
\begin{theorem}\label{esistenza_piastrelle}{\rm \cite{OW}}
Let $G$ be a finitely generated amenable group and let $\{F_n\}_{n\in\N}$ be a F\o lner exhaustion of $G$. Then, for all $\varepsilon\in(0,1/4)$ and $\bar n\in\N$, there exist positive integers $n_1,\dots,n_k$ such that $\bar n\leq n_1\leq \dots \leq n_k$ and $\{F_{n_1},\dots,F_{n_k}\}$ $\varepsilon$-quasi-tiles $F_m$, for any big enough $m$.
\end{theorem}

In the rest of the subsection we recall some results and terminology about non-negative invariants for the finite subsets of $G$, that is, functions $f : \F(G) \to \R_{\geq0}$. In particular, we say that $f$ is
\begin{enumerate}[\rm (1)]
\item {\em monotone} if $f(A)\leq f(A')$, for all $A\subseteq A'\in \F(G)$;
\item {\em sub-additive} if  $f(A\cup A')\leq f(A)+f(A')$, for all $A,$ $A'\in \F(G)$;
\item {\em (left) $G$-equivariant} if $f(gA)= f(A)$, for all $A\in \F(G)$ and $g\in G$.
\end{enumerate}
A consequence of (3) above is that $f(\{g\})=f(\{e\})$, for all $g\in G$. Thus by (2), $f(A)\leq\sum_{g\in A}f(\{g\})= |A|f(e)$, for all $A\in \F(G)$.

\medskip
The following result, generally known as ``Ornstein-Weiss Lemma", is proved in \cite{OW} for a suitable class of locally compact amenable groups (a direct proof, along the same lines, in the discrete case can be found in \cite{WZ}, while a nice alternative argument, based on ideas of Gromov, is given in \cite{Krieger}).

\begin{lemma}
Let $G$ be a finitely generated amenable group and consider a monotone, sub-additive and $G$-equivariant function $f:\F(G)\to \R_{\geq 0}$. Then, for any F\o lner sequence $\{F_n\}_{n\in\N}$, the sequence $(f(F_n)/|F_n|)_{n\in\N}$ converges and the value of the limit $\lim_{n\in \N}{f(F_n)}/{|F_n|}$
is the same for any choice of the F\o lner sequence.
\end{lemma}

We conclude this paragraph with the following consequence of Theorem \ref{esistenza_piastrelle}. The proof is essentially given in  \cite{HYG}, we give it here for completeness sake, as our statement is slightly different.

\begin{corollary}\label{scelta_unica_n}
Let $G$ be a finitely generated amenable group and $\{F_n\}_{n\in\N}$ be a F\o lner exhaustion of $G$. Then, for any $\varepsilon\in(0,1/4)$ and $\bar n\in\N$ there exist integers $n_1, \dots , n_k$ such that $\bar n \leq n_1  \leq \dots\leq n_k$ and, for any sub-additive and $G$-equivariant $f : \F(G) \to \R$ we have
$$\limsup_{n\to\infty}\frac{f(F_n)}{|F_n|}\leq M \varepsilon+\frac{1}{1 - \varepsilon}\cdot\max_{1\leq i\leq k} \frac{f(F_{n_i})}{|F_{n_i}|}\, ,$$
where $M = f(\{e\})$. 
\end{corollary}
It is important to underline that the choice of the $n_1, \dots , n_k$ does not depend on the function $f$, but we can really find a family $\{n_1, \dots , n_k\}$, which works for all $f$ at the same time.
\begin{proof}
Let $\varepsilon\in(0,1/4)$ and $\bar n\in\N$. By Theorem \ref{esistenza_piastrelle}, there exist positive integers $n_1,\dots,n_k$ such that $\bar n\leq n_1\leq \dots \leq n_k$ and $\{F_{n_1},\dots,F_{n_k}\}$ $\varepsilon$-quasi-tiles $F_n$, for all $n\geq \bar n$. We let $C_1^1,\dots, C_k^n$ be the tiling centers for $F_n$. Thus, when $n\geq \bar n$, we have
$$ F_n \supseteq \bigcup_{i=1}^k C_{i}^nF_{n_i} \ \ \text{ and } \ \  \left| \bigcup_{i=1}^k C_{i}^nF_{n_i}\right| \geq \max\left\{(1 - \varepsilon)|F_n|\ ,\ (1-\varepsilon)\sum_{i=1}^k|C^n_{i}| \cdot |F_{n_i}|\right\}\,.$$
Now, let $f:\F(G)\to \R_{\geq 0}$ be a sub-additive and $G$-equivariant function, we obtain that
\begin{align*}
\frac{f(F_n)}{|F_n|}&\leq \frac{f\left(F_n \setminus\bigcup_{i=1}^kC_{i}^nF_{n_i}\right)}{|F_n|} + \frac{f\left(\bigcup_{i=1}^k C_{i}^nF_{n_i}\right)}{|F_n|}\\
&\leq M\cdot \frac{\left|F_n \setminus\bigcup_{i=1}^kC_{i}^nF_{n_i}\right|}{|F_n|}+\frac{f\left(\bigcup_{i=1}^k C_{i}^nF_{n_i}\right)}{\left|\bigcup_{i=1}^k C_{i}^nF_{n_i}\right|}\\
&\leq M\varepsilon+\frac{f\left(\bigcup_{i=1}^k C_{i}^nF_{n_i}\right)}{\left|\bigcup_{i=1}^k C_{i}^nF_{n_i}\right|}\\
&\leq M\varepsilon+\frac{\sum_{i=1}^k |C_{i}^n|f(F_{n_i})}{(1-\varepsilon)\sum_{i=1}^k|C^n_{i}| \cdot |F_{n_i}|}\leq
M\varepsilon+\frac{1}{1-\varepsilon}\cdot\max_{1\leq i\leq k}\frac{f(F_{n_i})}{|F_{n_i}|}\,.
\qedhere\end{align*}
\end{proof}

\subsection{Crossed products}\label{crossed_subs}

Given a group $G$ and a ring $R$, a {\em crossed product} $\RG $ of $R$ with $G$ is a ring constructed as follows: as a set, $\RG $ is the collection of all the formal sums of the form
$$\sum_{g\in G}r_g\underline g\, ,$$
with $r_g\in R$ and $r_g=0$ for all but finite $g\in G$, and where the $\underline g$ are symbols uniquely assigned to each  $g\in G$. Sum in $\RG $ is as expected, that is, it is defined component-wise exploiting the addition in $R$:
$$\left(\sum_{g\in G}r_g\underline g\right)+\left(\sum_{g\in G}s_g\underline g\right)=\sum_{g\in G}(r_g+s_g)\underline g\, .$$
In order to define a product in $\RG $, we need to take two maps
$$\sigma:G\to \Aut_{ring}(R)\ \  \text{ and }\ \  \rho:G\times G\to U(R) \, ,$$
where $\Aut_{ring}(R)$ denotes the group of automorphisms of $R$ in the category of (unitary) rings, while $U(R)$ is the group of units of $R$. Given $g\in G$ and $r\in R$ we denote the image of $r$ via the automorphism $\sigma(g)$ by $r^{\sigma(g)}$. We suppose also that $\sigma$ and $\rho$ satisfy the following conditions for all $r\in R$ and $g_1,$ $g_2$ and $g_3\in G$:
\begin{enumerate}[\rm (Cross.1)]
\item $\rho(g_1,g_2)\rho(g_1g_2,g_3)=\rho(g_2,g_3)^{\sigma(g_1)}\rho(g_1,g_2g_3)$;
\item $r^{\sigma(g_2)\sigma(g_1)}=\rho(g_1,g_2)r^{\sigma(g_1g_2)}\rho(g_1,g_2)^{-1}$;
\item $\rho(g,e)=\rho(e,g)=1$ (for all $g\in G$) and $\sigma(e)=1$.
\end{enumerate}
The product in $\RG $ is defined by the rule $(r\underline g)(s\underline h)=rs^{\sigma(g)}\rho(g,h)\underline{gh}$, together with bilinearity, that is
$$
\left(\sum_{g\in G}r_g\underline g\right) \left(\sum_{g\in G}s_g\underline g\right)=\sum_{g\in G}\left(\sum_{h_1h_2=g}r_{h_1}s_{h_2}^{\sigma(h_1)}\rho(h_1,h_2)\right)\underline g\, .$$
By (Cross.1) and (Cross.2) above, $\RG $ is an associative and unitary ring, while by (Cross.3) $1_{\RG }=\underline e$. For more details on this kind of construction we refer to \cite{Passman}.

\medskip
Notice that there is a canonical injective ring homomorphism $R\to \RG$ such that $r\mapsto r\underline e$; for this we identify $R$ with a subring of $\RG $. This allows one to consider a forgetful functor from $\lmod {\RG }\to \lmod R$. 
 On the other hand, in general there is no natural map $G\to \RG $ which is compatible with the operations, anyway the obvious assignment $g\mapsto \underline g$ respects the operations {\em modulo some units of $R$}. Thus, given a left $\RG$-module $M$, there is a canonical map 
\begin{equation}\label{azione_brutta}\lambda:G\to \Aut_{\Z}(M)\ \,, \ \ \ g\mapsto \lambda_{g}\ \,, \ \ \ \lambda_g(m)=\underline g m\,,\end{equation}
which is not in general a homomorphism of groups. Given an $R$-submodule ${}_{R}K\leq M$ and an element $g\in G$, $\lambda_g(K)$ is again an $R$-submodule of $M$ but it is not in general isomorphic to ${}_RK$. As described in \eqref{eq_equiv_auto}, 
there is a self-equivalence of the category $\lmod R$, induced by the ring automorphism $\sigma(g)$
$$F_{\sigma(g)}\colon\lmod R\to \lmod R\, .$$
It follows by the definitions that $\lambda_g(K)= \underline g K\cong F_{\sigma(g)} K$. In particular, if $L$ is a length function compatible with the equivalence $F_{\sigma(g)}$, then $L(\lambda_g(K))=L(K)$. This useful fact motivates the following
\begin{definition}\label{crossed_comp}
Let $\RG$ be a crossed product and let $L\colon\lmod R\to \R_{\geq 0}\cup\{\infty\}$ be a length function. Then, $L$ is said to be {\em compatible with $\RG$} provided $L$ is compatible with $F_{\sigma(g)}$, for all $g\in G$.
\end{definition}

\section{The algebraic $L$-entropy}\label{amenable_ent}

We fix all along this section a ring $R$, an infinite finitely generated amenable group $G$, a crossed product $\RG$ and a length function $L\colon\lmod R\to \R_{\geq0}\cup\{\infty\}$ compatible with $\RG$. \\
Let us remark that, if $R$ is a (skew) field, $L=\dim_R$ is compatible with any crossed product $\RG$; more generally, this happens for all the functions $\ell_\alpha$ described in Subsection \ref{examples}. On the other hand, if $\RG=R[G]$, then any length function is trivially compatible with $\RG$.


\subsection{Definition and basic properties}
In this subsection we  define the algebraic entropy as an invariant for left $\RG$-modules. It turns out that this notion is not well-behaved on all the $\RG$-modules but just on a subclass of $\lmod \RG$ of all the $\RG$-modules $M$ for which the family of submodules $\Fin_L(M):=\{{}_RK\leq M:L(K)<\infty\}$ is big enough.
More precisely:
\begin{definition}
A left $R$-module $M$ is said to be {\em locally $L$-finite} if $\Fin_L(M)$ contains all the finitely generated submodules of $M$. We denote by $\lFin(L)$ the class of all the locally $L$-finite left $R$-modules while we let $\lFin_L(\RG)$ be the class of all the left $\RG$-modules ${}_{\RG}M$ such that ${}_RM\in \lFin(L)$, that is
$$\lFin_L(\RG):=\{M\in\lmod {\RG}: L(K)<\infty \text{ for any f.g. $_RK\leq M$}\}\,.$$ 
\end{definition}  
Notice that $\lFin(L)$ is closed under taking direct limits, quotients and submodules but not in general under taking extensions (see \cite{SVV}).

\medskip
Let $M\in\lFin_L(\RG)$. Then, given $K\in \Fin_L(M)$ and $F\in \F(G)$ we let
$$T_F(K):=\sum_{g\in F}\underline g K\,.$$
By the additivity of $L$ and the compatibility of $L$ with $\RG$, $L(T_F(K))\leq \sum_{g\in F}L(\underline g K)=|F|\, L(K)<\infty$. 
In particular, for any $L$-finite submodule $K$ of $M$ we obtain a function
$$f_K:\F(G)\to \R_{\geq 0}\ \ \ f_K(F):=L(T_{F}(K))\, .$$
We now verify that the above function satisfies the hypotheses of the Ornstein-Weiss Lemma:
\begin{lemma}\label{verifica_def}
In the above notation, $f_K$ is monotone, sub-additive and $G$-equivariant. 
\end{lemma}
\begin{proof}
Let $F,$ $F'\in\F(G)$. The proof follows by the properties of $L$, in fact, $T_{F}( K)+ T_{F'}( K)=T_{F\cup F'}( K)$ and so $f_K(F)+f_K(F')\geq f_K(F\cup F')$, proving sub-additivity. Furthermore, if $F\subseteq F'$ then $T_{F}(K)\leq T_{F'}(K)$ and so $f_K(F)\leq f_K(F')$, proving monotonicity. Finally, for any $g\in G$, 
$$T_{gF}(K)=\sum_{f\in F} \underline {gf}K=\sum_{f\in F} \underline {g}\ \underline{f}(\rho(g,f)^{\sigma(gh)})^{-1}K=\underline {g}\sum_{f\in F} \underline{f}K=\underline {g}T_F(K)\,,$$
since $(\rho(g,f)^{\sigma(gh)})^{-1}$ is a unit of $R$ and so $(\rho(g,f)^{\sigma(gh)})^{-1}K=K$.
Thus, by the compatibility of $L$ with $\RG$, $L(T_{gF}(K))=L(\underline gT_{F}(K))=L(T_{F}(K))$, proving that $f_K$ is $G$-equivariant.
\end{proof}

Thus, by the Ornstein-Weiss Lemma, the limit in the following definition exists and it does not depend on the choice of the F\o lner sequence.
\begin{definition}
Let $M\in\lFin_L(\RG)$, let $\{F_n\}_{n\in\N}$ be a F\o lner sequence for $G$ and let ${}_RK\in\Fin_L(M)$. The {\em $L$-entropy of ${}_{\RG}M$ with respect to $K$} is
$$\ent_L({}_{\RG}M,K):=\lim_{n\to\infty}\frac{L(T_{F_n}(K))}{|F_n|}\, .$$
The {\em $L$-entropy} of the $\RG$-module $_{\RG}M$ is $\ent_L(_{\RG}M):=\sup\{\ent_L(M,K):K\in\Fin_L(M)\}$.
\end{definition}

\begin{remark}
We defined the $L$-entropy for left $\RG$-modules in case $G$ is finitely generated. Anyway, the exact same procedure allows one to define this invariant when $G$ is just countable (but not necessarily finitely generated). Furthermore, standard variations of the above arguments using F\o lner nets allow one to define a similar invariant in case $G$ is not countable. 
\end{remark}

\begin{example}\label{ent_fin=0}
Given ${}_{\RG}M\in\lFin_L(\RG)$ such that $L({}_RM)<\infty$, $\ent_L({}_{\RG }M)=0$. Indeed, given $K\in\Fin_L(M)$, $\ent_L(M,K)\leq \lim_{n\to\infty} L(M)/|F_n|\leq \lim_{n\to\infty} L(M)/n=0$ (use the fact that, as $G$ is infinite, we can take a F\o lner sequence such that $F_n\lneq F_{n+1}$ for all $n\in\N$, thus $|F_n|\geq n$).
\end{example}

Given a locally $L$-finite $\RG$-module ${}_{\RG}M$ and an $\RG$-submodule $N\leq M$, there is an inclusion $\Fin_L(N)\subseteq \Fin_L(M)$, this easily implies the following lemma:
\begin{lemma}\label{monotonie}
Let ${}_{\RG}M\in\lFin_L(\RG)$ and ${}_\RG N\leq M$. Then, $\ent_L(_{\RG }M)\geq \ent_L(_{\RG }N)$.
\end{lemma}

In fact, entropy is also monotone under taking quotients, but we need some more work before showing that (see Corollary \ref{monotonia_quozienti}). Let us conclude this subsection with the following consequence of Corollary \ref{scelta_unica_n}:

\begin{corollary}\label{leq_max}
Let $\{F_n\}_{n\in\N}$ be a F\o lner exhaustion of $G$. Then, for any $\varepsilon \in (0,1/4)$ and $n\in \N$ there exist $n_1,\dots , n_k\in\N$ with $n\leq n_1\leq \dots\leq n_k$ such that, given an $L$-finite submodule $K\leq M$, 
$$\ent_L(M,K)\leq \varepsilon \cdot L(K)+\frac{1}{1-\varepsilon}\cdot \max_{1\leq i\leq k} \frac{L(T_{F_{n_i}}(K))}{|F_{n_i}|}\, .$$
\end{corollary}
\begin{proof}
This is a straightforward application of Corollary \ref{scelta_unica_n}. In fact, the function $f_{K}:\F(G)\to \R_{\geq 0}$ such that $f_K(F)=L(T_{F}(K))$ satisfies the hypotheses of such corollary for any $L$-finite $R$-submodule $K$ of $M$, by Lemma \ref{verifica_def}. Furthermore, $\lim_{n\to\infty}f_K({F_n})/|F_n|=\ent_L(M,K)$ by the definition of entropy.
\end{proof}

\subsection{The algebraic entropy is upper continuous}\label{upp_cont_subsec}

The following result allows us to redefine the algebraic entropy in terms of finitely generated submodules.

\begin{proposition}\label{upp-cont}
Let ${}_{\RG}M\in\lFin_L(\RG)$ and let $H\leq K\in \Fin_L(M)$. 
\begin{enumerate}[\rm (1)]
\item if $L(K/H)<\varepsilon$, then $\ent_L(M,K)-\ent_L(M,H)<\varepsilon$;
\item $\ent_L(_{\RG }M)=\sup\{\ent_L(M,K):{_RK\text{ finitely generated}}\}$.
\end{enumerate}
\end{proposition}
\begin{proof}
For any $g\in G$,  $\underline g K/\underline g H\cong F_{\sigma( g)}(K/H)$ so, since $L$ is compatible with $F_{\sigma(g)}$, $L(\underline g K/\underline g H)<\varepsilon$. Thus, by the additivity of $L$, 
$$L(T_{F_n}(K)/T_{F_n}(H))\leq \sum_{g\in F_n}L((\underline g K+T_{F_n}(H))/T_{F_n}(H))< |F_n|\varepsilon\, ,$$
for all $n\in \N$, where $\{F_n\}_{n\in\N}$ is a F\o lner sequence. Therefore, $\ent_L(M,K)-\ent_L(M,H)<\varepsilon$. We can now verify part (2). Indeed, the  inequality ``$\leq$" comes directly from the definition of entropy. On the other hand, given $K\in \Fin_L(M)$ and $\varepsilon>0$, by the upper continuity of $L$ there exists $H\leq K$ finitely generated such that $L(K)-L(H)<\varepsilon$. By part (1), $\ent_L(M,K)<\ent_L(M,H)+\varepsilon$, which easily yields the claim.
\end{proof}

The following corollary deals with the case when $M$ is  generated (as $\RG$-module) by an $L$-finite $R$-submodule $K$, that is, $M=T_G(K)$. In such situation one does not need to take a supremum to compute entropy. 

\begin{corollary}\label{fin}
Let $M$ be a left $\RG$-module such that $M=T_G(K)$ for some $K\in\Fin_L(M)$, then $$\ent_L(M)=\ent_L(M,K)\,.$$
\end{corollary}
\begin{proof}
Given a finitely generated $R$-submodule $H$ of $M$, we can find a finite subset $e\in F\subseteq G$ such that $H\subseteq T_F(K)$. This shows that $\ent_L(M,H)\leq \ent_L(M,T_F(K))$. By the F\o lner condition, 
\begin{equation*}\label{Bernoulli_folner}
\lim_{n\to\infty}\frac{|F_nF|}{|F_n|}\leq
\lim_{n\to\infty}\frac{|F_n\cup\bigcup_{f\in F}\partial_F( F_n)f|}{|F_n|}
\leq 1+\lim_{n\to\infty}\sum_{f\in F}\frac{|\partial_F( F_n)f|}{|F_n|}=1\, .\end{equation*}
On the other hand, ${|F_nF|}/{|F_n|}\geq 1$ so $\lim_{n\to\infty}{|F_nF|}/{|F_n|}=1$. We obtain that
$$\ent_L(M,T_F(K))=\lim_{n\to\infty}\frac{T_{F_n}(T_F(K))}{|F_n|}=\lim_{n\to\infty}\frac{T_{F_nF}(K)}{|F_n|}\cdot\frac{|F_n|}{|F_nF|}=\ent_L(M,K)\, ,
$$
where the last equality comes from the fact that $\{F_nF\}_{n\in\N}$ is a F\o lner sequence (use the fact that, for any $C\subseteq \F(G)$ and $n\in\N$, one has the inclusion $\partial_{C}(F_nF)\subseteq \partial_{CF^{-1}}(F_n)$ and apply the F\o lner condition for $\{F_n\}_{n\in\N}$) and the definition of $\ent_L$ does not depend on the choice of a particular F\o lner sequence. Thus, $\ent_L(M,H)\leq \ent_L(M,K)$ for any finitely generated $H\leq M$; one concludes by Proposition \ref{upp-cont}.
\end{proof}

The upper continuity of $\ent_L$ can now be verified easily using the above lemma and Proposition \ref{upp-cont}:

\begin{corollary}\label{amen_upp}
The invariant $\ent_L:\lFin_L(\RG)\to\R_{\geq 0}\cup\{\infty\}$ is upper continuous. 
\end{corollary}
\begin{proof}
The fact that $\ent_L$ is an invariant can be derived by the definition and the fact that $L$ is an invariant. Now, let $M\in \lFin_L(\RG)$, then by Proposition \ref{upp-cont} and Corollary \ref{fin} we get
\begin{align*}
\ent_L(_{\RG}M)&=\sup\{\ent_L(M,K):{K\text{ finitely generated $R$-submodule of $M$}}\}\\
&=\sup\{\ent_L(_{\RG}(T_G(K)):{K\text{ finitely generated $R$-submodule of $M$}}\}\\
&=\sup\{\ent_L(_{\RG}N):{N\text{ finitely generated $\RG$-submodule of $M$}}\}\, .
\qedhere\end{align*}
\end{proof}

Another consequence of Proposition \ref{upp-cont} is the monotonicity of entropy under taking quotients:

\begin{corollary}\label{monotonia_quozienti}
Let ${}_{\RG}M\in\lFin_L(\RG)$ and ${}_{\RG} N\leq M$. Then, $\ent_L(_{\RG }M)\geq \ent_L(_{\RG }(M/N))$.
\end{corollary}
\begin{proof}
Given a finitely generated submodule $\bar K\leq M/N$, there exists a finitely generated (thus $L$-finite) submodule $K\leq M$ such that $(K+N)/N\cong \bar K$. Given a F\o lner sequence $\{F_n\}_{n\in\N}$,
$$T_{F_n}(\bar K)=(T_{F_n}(K)+N)/N\ \text{ and so } L(T_{F_n}(\bar K))\leq L(T_{F_n}(K))\ \text{ for all $n\in\N$}\, .$$ 
Dividing by $|F_n|$ and passing to the limit we get $\ent_L(M/N,\bar K)\leq \ent_L(M,K)\leq \ent_L(_{\RG }M)$.
\end{proof}

We conclude this subsection with the following property that will be extremely useful in proving the additivity of entropy.

\begin{proposition}\label{miracolo}
Let $M\in\lFin_L(\RG)$ be finitely generated and let $\{N_i:i\in I\}$ be a directed system of $\RG$-submodules. Letting $N:=\bigcup_iN_i$,   $\ent_L(M/N)=\inf_{i}\ent_L(M/N_i)$.
\end{proposition}
\begin{proof}
The inequality ``$\leq$" follows by Corollary \ref{monotonia_quozienti}. On the other hand, let $K\in \Fin_L(M)$ be a finitely generated $R$-submodule such that $T_G(K)=M$; let also $\bar K:=(K+N)/N$ and $\bar K_i:=(K+N_i)/N_i$ for all $i\in I$. By Lemma \ref{fin}, $\ent_L(M/N)=\ent_L(M/N,\bar K)$ and $\ent_L(M/N_i)=\ent_L(M/N_i,\bar K_i)$, for all $i$.
Now fix (arbitrarily) a constant $\varepsilon\in (0,1/4)$ and notice that:
\begin{enumerate}[(a)] 
\item there exists $m\in\N$ such that  $\displaystyle	\frac{L(T_{F_n}(\bar K))}{|F_n|}\leq \ent_L(M/N)+\varepsilon$ for all $n\geq m$;
\item there exist $m\leq m_1\leq \dots\leq  m_k\in\N$ such that
$\displaystyle	\ent_L(M/N_i)\leq \varepsilon \cdot L(\bar K_i)+\frac{1}{1-\varepsilon}\cdot \max_{1\leq t\leq k} \frac{L(T_{F_{m_t}}(\bar K_i))}{|F_{m_t}|}$, for all $i\in I$, by Corollary \ref{leq_max};
\item for any $t\in \{1,\dots,k\}$ there exists $j_t\in I$ such that 
$\displaystyle	L(T_{F_{m_t}}(\bar K_{j}))\leq \inf_i L(T_{F_{m_t}}(\bar K_i)) +\varepsilon=L(T_{F_{m_t}}(\bar K)) +\varepsilon$, whenever $N_j\geq N_{j_t}$, where for the equality we use Proposition \ref{add_on_fg}.
\end{enumerate}
Since $\{N_i:i\in I\}$ is directed, there exists $\bar j\in I$ such that $\sum_{t=1}^kN_{j_t}\leq N_{\bar j}$. Thus,
\begin{align*}
\inf_{i}\ent_L(M/N_i)\leq \ent_L(M/N_{\bar j})&
\overset{(b)}\leq\varepsilon \cdot L(\bar K_{\bar j})+\frac{1}{1-\varepsilon}\cdot \max_{1\leq t\leq k} \frac{L(T_{F_{m_t}}(\bar K_{\bar j}))}{|F_{m_t}|}\\
&\overset{(c)}\leq \varepsilon \cdot L(\bar K_{\bar j})+\frac{1}{1-\varepsilon}\cdot \frac{L(T_{F_{m_t}}(\bar K))+\varepsilon}{|F_{m_t}|}\\
&\overset{(a)}\leq \varepsilon \cdot L(\bar K_{\bar j})+\frac{1}{1-\varepsilon}\cdot (\ent_L(M/N)+2\varepsilon)\, .
\end{align*}
As this holds for any small enough $\varepsilon>0$, the conclusion follows. 
\end{proof}

\subsection{The algebraic entropy is additive}\label{AT_subsec}

In this subsection we complete the proof of the fact that 
$$\ent_L\colon \lFin_L(\RG)\to \R_{\geq0}\cup\{\infty\}$$
is a length function. In fact we have already seen that $\ent_L$ is an upper continuous invariant. To verify the additivity of $\ent_L$, by Propositions \ref{add_on_fg} and \ref{miracolo}, it is enough to check that $\ent_L$ is additive on short exact sequences of finitely generated modules.\\ The computations in the following proposition are freely inspired to the proof of the Abramov-Rokhlin Formula given in \cite{WZ}. The context (and even the statements) in that paper is quite different but the ideas contained there can be perfectly adapted to our needs.

\begin{proposition}
Let $0\to N\to M\to M/N\to 0$ be a short exact sequence of finitely generated left $\RG$-modules, then $\ent_L\left({_{\RG}M}\right)= \ent_L\left({_{\RG}N}\right)+\ent_L\left({_{\RG}(M/N)}\right)$.
\end{proposition}
\begin{proof}
Fix a F\o lner exhaustion $\{F_n\}_{n\in\N}$ of $G$ and finitely generated $R$-submodules $K\leq M$ and $K'\leq K\cap N$ such that $T_G(K')=N$ and $T_G(K)=M$. If we let $\bar K:=(K+N)/N\leq M/N$, then $T_G(\bar K)=M/N$. By Lemma \ref{fin}, $\ent_L(M)=\ent_L(M, K)$,  $\ent_L(N)=\ent_L(N, K')$ and $\ent_L(M/N)=\ent_L(M/N, \bar K)$.\\ Let $\varepsilon\in(0,1/4)$, by the existence of the limits defining $L$-entropy, we can find $\bar n\in\N$ such that
\begin{align*}
\left|\frac{L(T_{F_n}(K))}{|F_n|}-\ent_L(M)\right|<\varepsilon\, , && \left|\frac{L(T_{F_n}(K'))}{|F_n|}-\ent_L(N)\right|<\varepsilon\, , && \left|\frac{L(T_{F_n}(\bar K))}{|F_n|}-\ent_L(M/N)\right|<\varepsilon\, ,
\end{align*} 
for all $n\geq\bar n$. Moreover, $L(T_{F_m}(K))= L(T_{F_m}( K'))+L(T_{F_m}(K)/T_{F_m}(K'))\geq L(T_{F_m}( K'))+L(T_{F_m}(\bar K))$, for all $m\in\N$,
and so, for all $m\geq \bar n$,
$$\ent_L(M)\geq \frac{L(T_{F_m}(K))}{|F_m|}-\varepsilon\geq   \frac{L(T_{F_m}(K'))}{|F_m|}+ \frac{L(T_{F_m}(\bar K))}{|F_m|}-\varepsilon\geq \ent_L(N)+\ent_L(M/N)-3\varepsilon\,,$$
yielding the inequality $\ent_L(M)\geq \ent_L(N)+\ent_L(M/N)$. For the converse inequality, notice that
\begin{align*}
\ent_L(M)&\leq \frac{L(T_{F_m}(K))}{|F_m|}+\varepsilon= \frac{L(T_{F_m}(K'))}{|F_m|}+ \frac{L(T_{F_m}(K)/T_m(K'))}{|F_m|}+\varepsilon\\
\notag&\leq \ent_L(N)+ \frac{L(T_{F_m}(K)/T_m(K'))}{|F_m|}+2\varepsilon\,.\end{align*}
To conclude we should verify that $(L(T_{F_m}(K)/T_m(K')))/|F_m|$ is close enough to $\ent_L(M/N)$. \\
Since $\{F_n\}$ is a F\o lner exhaustion, $N=\bigcup_{n\in\N}T_{F_n}(K')$ and so, for any $L$-finite submodule $H\leq M$, $L(H\cap N)=\lim_{n\to\infty}L(H\cap T_{F_n}(K'))$, by upper continuity. Using additivity, $L((H+N)/N)=\lim_{n\to\infty}L((H+T_{F_n}(K'))/T_{F_n}(K'))$, so that  $|L((H+N)/N)-L((H+T_{F_n}(K'))/T_{F_n}(K'))|<\varepsilon$,
for any big enough $n\in\N$. By Theorem \ref{esistenza_piastrelle}, there exist $\bar n<n_1<\dots<n_k\in \N$ and $\bar n'$ such that $\{F_{n_1},\dots,F_{n_k}\}$ $\varepsilon$-quasi-tiles $F_m$ for all $m\geq \bar n'$. Applying the above argument with $H=T_{F_{n_i}}(K)$ (for all $i=1,\dots, k$), we can find $\overline{\overline{n}}\geq \max\{\bar n',n_k\}$ such that
\begin{equation}\label{stima_1}\left|L\left(\frac{T_{F_{n_i}}(K)+T_{F_n}(K')}{T_{F_n}(K')}\right)- L(T_{F_{n_i}}(\bar K))\right|<\varepsilon\end{equation}
for all $n\geq \overline{\overline{n}}$ and all $i=1,\dots,k$. From now on we fix $m\geq \bar {\bar n}$ such that
$ |\partial_{F_{\bar{\bar{n}}}^{-1}F_{\bar{\bar{n}}}}(F_m)|/|F_m|\leq \varepsilon$.
Notice that, given $i\in\{1,\dots,k\}$, $F_{n_i}\subseteq F_{\bar{\bar n}}$, and so $\partial_{F_{\bar{\bar{n}}}}(F_m)F_{n_i}\subseteq \partial_{F_{\bar{\bar{n}}}^{-1}F_{\bar{\bar{n}}}}(F_m)$. Thus, 
$$|\partial_{F_{\bar{\bar{n}}}}(F_m)F_{n_i}|/|F_m|<\varepsilon\,, \ \ \text{ for all $i=1,\dots,k.$}$$
We choose tiling centers $C_1,\dots,C_k$ for $F_m$, so that
\begin{equation}\label{stime_2}
|F_m|\geq \left|\bigcup_{i=1}^kC_iF_{n_i}\right|\geq \max\left\{(1-\varepsilon)|F_m|\ ,\ (1-\varepsilon)\sum_{i=1}^k|C_i||F_{n_i}|\right\}\, ,
\end{equation}
Consider the following estimate:
\begin{align}\label{stima_nuova_1}
\frac{1}{|F_m|}L\left(\frac{T_{F_m\setminus \bigcup_{i}C_iF_{n_i}}(K)+T_{F_m}(K')}{T_{F_m}(K')}\right)&\leq \frac{L(T_{F_m\setminus \bigcup_{i}C_iF_{n_i}}(K))}{|F_m|}\\
\notag&\leq \frac{|F_m\setminus \bigcup_{i}C_iF_{n_i}|}{|F_m|}L(K)\leq \varepsilon L(K)\,,\end{align}
Given $i\in\{1,\dots,k\}$, using that $C_iF_{n_i}\subseteq (C_i\setminus \partial_{F_{\overline{\overline{n}}}}(F_m))F_{n_i}\cup \partial_{F_{\overline{\overline{n}}}}(F_m)F_{n_i}$, we get:
{\footnotesize\begin{align}\label{ultima_approx_AT}
\notag \frac{1}{|F_m|}
&L\left(
\frac{T_{C_iF_{n_i}}(K)+T_{F_m}(K')}{T_{F_m}(K')}\right)\leq\\
\notag&\leq \frac{1}{|F_m|}
\sum_{c\in C_i\setminus\partial_{F_{\overline{\overline{n}}}}(F_m)}
L\left(
\frac{T_{cF_{n_i}}(K)+T_{F_m}(K')}{T_{F_m}(K')}\right)+\frac{1}{|F_m|}L\left(\frac{T_{\partial_{F_{\overline{\overline{n}}}}(F_m)F_{n_i}}(K)+T_{F_m}(K')}{T_{F_{m}}(K')}\right)\\
\notag&{\leq} \frac{1}{|F_m|}\sum_{c\in C_i\setminus\partial_{F_{\overline{\overline{n}}}}(F_m)}L\left(\frac{T_{F_{n_i}}(K)+T_{c^{-1}F_m}(K')}{T_{c^{-1}F_m}(K')}\right)+\frac{| \partial_{F_{\overline{\overline{n}}}}(F_m)F_{n_i}|L(K)}{|F_m|}\\ 
\notag&\leq \frac{|C_i|}{|F_{m}|}L\left(\frac{T_{F_{n_i}}(K)+T_{F_{\overline{\overline{n}}}}(K')}{T_{F_{\overline{\overline{n}}}}(K')}\right)+\varepsilon L(K)\\
&\leq \frac{|C_i|L(T_{F_{n_i}}(\bar K))+|C_i|\varepsilon}{|F_{m}|}+\varepsilon L(K)\leq \frac{|C_i||F_{n_i}|}{|F_{m}|}\ent_L(M/N)+\varepsilon\left(\frac{|C_i||F_{n_i}|}{|F_{m}|}+\frac{|C_i|}{|F_{m}|}+L(K)\right)
\end{align}}
where the second inequality follows by the compatibility of $L$ with $\RG$, the third one is true since $F_{\overline{\overline{n}}}\subseteq c^{-1}F_m$ for all $c\in C_i\setminus\partial_{F_{\overline{\overline{n}}}}(F_m)$, and the fourth one follows by \eqref{stima_1}.\\
Let us assemble  together the above computations:
{\footnotesize\begin{align*}
\frac{L(T_{F_m}(K)/T_{F_m}(K'))}{|F_m|}&\leq 
\frac{1}{|F_m|}
L\left(
\frac{T_{\bigcup_{i}C_iF_{n_i}}(K)+T_{F_m}(K')}{T_{F_m}(K')}\right)
+\frac{1}{|F_m|}L\left(
\frac{T_{F_m\setminus \bigcup_{i}C_iF_{n_i}}(K)+T_{F_m}(K')}{T_{F_m}(K')}\right)\\
&\overset{\eqref{stima_nuova_1}}{\leq} 
\frac{1}{|F_m|}\sum_{i=1}^kL\left(
\frac{T_{C_iF_{n_i}}(K)+T_{F_m}(K')}{T_{F_m}(K')}\right)+\varepsilon L(K)\\
&\overset{\eqref{ultima_approx_AT}}{\leq} \sum_{i=1}^k\left(\frac{|C_i||F_{n_i}|}{|F_{m}|}\ent_L(M/N)+\varepsilon\left(\frac{|C_i||F_{n_i}|}{|F_{m}|}+\frac{|C_i|}{|F_{m}|}+L(K)\right)\right)
+\varepsilon L(K)\\
&\leq \frac{\ent_L(M/N)}{1-\varepsilon}+\varepsilon \left(\frac{2}{1-\varepsilon}+(k+1)L(K)\right)\,.
\end{align*}  }
Hence, we have obtained that
$$\ent_L(M)\leq \ent_L(N)+\frac{\ent_L(M/N)}{1-\varepsilon}+\varepsilon \left(\frac{2}{1-\varepsilon}+(k+1)L(K)\right)\,,$$
holds for any $\varepsilon\in(0,1/4)$, and so $\ent_L(M)\leq \ent_L(N)+\ent_L(M/N)$, concluding the proof.
\end{proof}

\subsection{Values on (sub)shifts}\label{shift_subsec}

This subsection is devoted to compute the values of the entropy on the left $\RG$-modules of the form $M=\RG\otimes_R K$, for some left $R$-module $K$, and their $\RG$-submodules. This will conclude the proof of Theorem B. 

\begin{definition}
A left $\RG$-module of the form $M=\RG\otimes_RK$ is said to be a {\em Bernoulli shift} while any of its $\RG$-submodules is a {\em subshift}.
\end{definition}

Notice that, given a Bernoulli shift $M=\RG\otimes_RK$, there is a direct sum decomposition, as a left $R$-module, $_{R}M\cong \bigoplus_{g\in G}\underline gK$. Thus, for any $F\in \F(G)$, 
\begin{equation}\label{prop_bernoulli}T_F(\underline gK)=\bigoplus_{h\in F}\underline {hg} K\ \ \ \text{ and }\ \ \ T_G(\underline g K)=M\,.\end{equation}

In the following example we compute the algebraic entropy of Bernoulli shifts.
\begin{example}\label{Bernoulli}
In the above notation, suppose $L(K)<\infty$. By Lemma \ref{fin} and \eqref{prop_bernoulli}, we obtain that $\ent_L({_{\RG}M})=\ent_L(M,K)$. Furthermore, again by \eqref{prop_bernoulli}, $L(T_F(K))/|F|=L(K)$,  for all $F\in \F(G)$. Therefore, $\ent_L(M)=L(K)$.
\end{example}

The computation in the above example shows that the entropy of $M=\RG\otimes_RK$ is $0$ if and only if $L(K)=0$, if and only if $L(M)=0$. Our next goal is to show that, if $_{\RG}N$ is a subshift of $M$, then $\ent_L(_{\RG}N)=0$ if and only if $L({}_RN)=0$. This will be proved in Proposition \ref{ent>>0} but first we need to recall some useful terminology and results from \cite{Silberstein}.
\begin{definition}
Let $E$ and $F$ be subsets of $G$. A subset ${\cal N} \subseteq G$ is an {\em $(E, F )$-net} if it satisfies the following conditions:
\begin{enumerate}[\rm (1)]
\item the subsets $(gE)_{g\in \cal N}$ are pairwise disjoint, that is, $gE \cap g'E = \emptyset$ for all $g\neq  g'\in {\cal N}$;
\item $G =\bigcup_{g\in {\cal N}}gF$.
\end{enumerate}
\end{definition}

The following lemma is a variation of \cite[Lemmas 2.2 and 4.3]{Silberstein}.
%

\begin{lemma}\label{AAAAAAAA}
Let $\{F_n\}_{n\in\N}$ be a F\o lner sequence, $E\subseteq F\subseteq G$ be finite subsets with $e\in F$ and $\cal N$ an $(E,F)$-net. Then,
\begin{enumerate}[\rm (1)]
\item  there exists an $(E,EE^{-1})$-net;
\item there exist $\alpha\in (0,1]$ and $n_0\in\N$ such that $|F_n\cap {\cal N}|\geq \alpha\cdot |F_n|$, for all $n>n_0$.
\end{enumerate}
\end{lemma}
\begin{proof}
Part (1) is proved in \cite[Lemma 2.2]{Silberstein}, so let us concentrate on part (2). For each $n\in\N$, let $F_n^{+F}=\Out_F(F_n)\cap \cal N$ and notice that $F_n^{+F}\setminus (F_n\cap {\cal N}) \subseteq \partial_F(F_n)$. Furthermore, since $F_n$ is covered by the sets $gF$, $g \in F_n^{+F}$, we have $|F_n| \leq |F| á |F_n^{+F}|$. Let now $\alpha_1=1/|F|$, thus
\begin{align*}
\alpha_1|F_n|-|F_n\cap {\cal N}|\leq |F_n^{+F}|-|F_n\cap {\cal N}|\leq |F_n^{+F}\setminus (F_n\cap {\cal N})|\leq \partial_F(F_n)\, .
\end{align*}
Let $\alpha_2\in (0,\alpha_1)$. By the F\o lner condition, there exists $n_0\in\N$ such that $|\partial_F(F_n)|/|F_n|\leq \alpha_2$ for all $n>n_0$. Thus, letting $\alpha=\alpha_1-\alpha_2\in (0,1]$, we get $|F_n\cap {\cal N}|\geq \alpha_1|F_n|- \partial_F(F_n)\geq \alpha |F_n|$, for all $n>n_0$.
\end{proof}

\begin{proposition}\label{ent>>0}
Let $K$ be an $L$-finite left $R$-module and $_{\RG}N$  a subshift of $M=\RG\otimes_RK$. Then, 
$$\ent_L({}_{\RG}N)=0\  \ \text{ if and only if }\ \ L(_RN)=0\,.$$
\end{proposition}
\begin{proof}
Suppose $L(_RN)\neq0$, then there exists $x\in N$ such that $L(Rx)\neq 0$. Let $E$ be the set of all elements $h\in G$ such that, writing $x=\sum_{g\in G}\underline g x_g$, the component $x_h$ is not $0$. We fix an $(E,EE^{-1})$-net $\cal N$. Notice  that, given $f_1\neq f_2\in \cal N$, then $\beta_{f_1}(Rx)\cap \beta_{f_2}(Rx)=0$. Thus, by Lemma \ref{AAAAAAAA}, we can find $n_0\in\N$ and $\alpha\in(0,1)$ such that
$$L(T_{F_n}(Rx))\geq L(T_{F_n\cap {\cal N}}(Rx))=|F_n\cap {\cal N}|L(Rx)\geq \alpha|F_n|L(Rx)$$
for all $n>n_0$. In particular, $\ent_L(_{\RG}N)\geq \ent_L(N,Rx)\geq \alpha L(Rx)\neq 0$.
\end{proof}

\section{Applications}\label{Kap_and_rel}

\subsection{Stable finiteness}
Let $R$ be a ring. Recall from the introduction that a left $R$-module $M$ is said to be {\em hopfian} if any of its surjective endomorphisms is bijective. 
 Recall also that a ring $R$ is {\em directly finite} if $xy=1$ implies $yx=1$ for all $x,y\in R$. Furthermore, $R$ is {\em stably finite} if the ring  $\Mat_k(R)$ of $k\times k$ square matrices with coefficients in $R$, is directly finite for all $k\in \N_+$.\\ The following lemma gives a natural connection between stable finiteness and Hopficity:
\begin{lemma}\label{hopf=df}
Given a ring $R$ and a positive integer $k$, the following are equivalent:
\begin{enumerate}[\rm (1)]
\item $R^k_R$ is hopfian (as a right $R$-module);
\item ${}_RR^k$ is hopfian (as a left $R$-module);
\item $\Mat_k(R)$ is directly finite.
\end{enumerate}
\end{lemma}
%
%

As we said, our interest in Hopficity is motivated by the Stable Finiteness Conjecture, stating that for any field $\K$ and any group $G$, the group ring $\K[G]$ is stably finite. 
%
%
%
%
%
%
%
In fact, a fairly general case of this conjecture was recently verified by Elek and Szab\'o \cite{Elek} that proved that $\K[G]$ is stably finite for any division ring $\K$ and any sofic group $G$. A straightforward consequence is that, under these hypotheses, $\Mat_{n}(\K[G])$ is stably finite. Now, by the Artin-Wedderburn Theorem, given a semisimple Artinian ring $R$, there exist positive integers $k,$ $n_1,\dots,n_k\in\N_+$ and division rings $\K_1,\dots,\K_{k}$ such that $R\cong \Mat_{n_1}(\K_1)\times\dots\times \Mat_{n_k}(\K_k)$. This implies that, $R[G]\cong \Mat_{n_1}(\K_1[G])\times\dots\times \Mat_{n_k}(\K_k[G])$, thus a consequence of the above theorem is that $R[G]$ is stably finite whenever $R$ is semisimple Artinian and $G$ sofic. This result can be further generalized as follows:

\begin{remark}\label{ferran}[Ferran Ced\'o, private communication (2012)]
If $R$ is a ring with left Krull dimension (for example see \cite{McRobson}) and $G$ is sofic, then $R[G]$ is stably finite.  First of all, notice that, if $I$ is a nilpotent ideal of $R$, then $I[G]=R[G]I$ is a nilpotent ideal of $R[G]$ and so one can reduce the problem modulo nilpotent ideals. Now, by \cite[Corollary 6.3.8]{McRobson}, the prime radical $N$ of $R$ is nilpotent and $N=P_1\cap\dots\cap P_m$, where $P_1,\dots,P_m$ are minimal prime ideals. Thus, by \cite[Proposition 6.3.5]{McRobson}, $R/N$ is a semiprime Goldie ring and so, by \cite[Theorem 2.3.6]{McRobson}  $R/N$ has a classical semisimple Artinian ring of quotients $S$. In particular, $(R/N)[G]$ embeds in $S[G]$ and  it is therefore stably finite. 
\end{remark}

Both the proof of the residually amenable case due to Ara, O'Meara and Perera, and the proof of the sofic case due to Elek and Szab\'o, consist in finding a suitable embedding of $\K[G]$ in a ring which is known to be stably finite. Such methods are really effective but, as far as we know, cannot be used to obtain information on the modules over $\K[G]$. It seems natural to ask the following
\begin{question}\label{quest_hopf}
Let $G$ be a group, $R$ a ring and let $\RG$ be a fixed crossed product. Given a finitely generated left $\RG$-module ${}_{\RG}M$, when is $_{\RG}M$ an hopfian module?
\end{question}

Using the theory of algebraic entropy we can now prove that a large class of left $\RG$-modules is hereditarily Hopfian (i.e., any submodule is Hopfian), in case $R$ is left Noetherian and $G$ amenable (see the statement of Theorem A). We remark that this is a very strong version of Kaplasky's Stable Finiteness Conjecture in the amenable case, which can be re-obtained as a corollary. The proof of Theorem A makes use of the full force of the localization techniques introduced in Section \ref{LF}. Such heavy machinery hides in some sense the idea behind the proof;  this is the reason for which we prefer to give first the proof of the following more elementary statement, whose proof is far more transparent:

\begin{lemma}
Let $\K$ be a division ring, let $G$ be a finitely generated amenable group and fix a crossed product $\K\asterisk G$. For all $n\in\N_+$, $(\K\asterisk G)^n=\K\asterisk G\otimes \K^n$ is a hereditarily hopfian left $\K\asterisk G$-module.
\end{lemma}
\begin{proof}
Let $n\in\N_+$ and choose $\K\asterisk G$-submodules $N\leq M\leq (\K\asterisk G)^n$ such that there exists a short exact sequence
$$0\to N\to M\to M\to 0\, ,$$
we have to show that $N=0$. The length function $\dim:\lmod \K\to \R_{\geq0}\cup\{\infty\}$ is compatible with any crossed product, so we can consider the $\dim$-entropy of left ${\K\asterisk G}$-modules. 
In particular, 
$$\ent_{\dim}(M)=\ent_{\dim}(M)+\ent_{\dim}(N)\ \ \ \text{ and }\ \ \ 0\leq\ent_{\dim}(N)\leq \ent_{\dim}(M)\leq \ent_{\dim}(\K\asterisk G\otimes\K^n)=n\,.$$ 
Thus, $\ent_{\dim}(N)=0$. By Proposition \ref{ent>>0}, this implies that $\dim (N)=0$, that is, $N=0$.
\end{proof}

The same argument of the above proof can be used to prove Theorem A, modulo the fundamental tool of Gabriel dimension:

\begin{proof}[Proof of Theorem A]
Consider a left $\RG$-submodule $M\leq \RG\otimes_RK$ and a short exact sequence of left $\RG$-modules
$$0\to \ker(\phi)\to M\overset{\phi}{\longrightarrow} M\to 0\, .$$
In order to go further with the proof we need to show that, as a left $R$-module, the Gabriel dimension of $\ker(\phi)$ is a successor ordinal whenever it is not $-1$ (i.e., whenever $\ker(\phi)\neq 0$). This follows by the following 
\begin{enumerate}
\item[] \begin{lemma}
In the hypotheses of Theorem A, $\Gdim({}_RN)$ is a successor ordinal for any non-trivial $R$-submodule $N\leq \RG\otimes_RK$.
\end{lemma}
\begin{proof}
A consequence of Lemma \ref{pre1} (6) is that $\tor_{\alpha+1}(K)/\tor_\alpha(K)\neq 0$ for just finitely many ordinals $\alpha$. Notice  that  $\tor_{\alpha}(\RG\otimes_RK)\cong \RG\otimes_R\tor_\alpha(K)$, as left $R$-modules, for any ordinal $\alpha$. Thus, $\tor_{\alpha+1}(\RG\otimes_RK)/\tor_\alpha(\RG\otimes_RK)\neq 0$ for finitely many ordinals. Notice also that $\tor_\alpha(N)=\tor_\alpha(\RG\otimes_RK) \cap N$ for all $\alpha$, thus, 
\begin{align*}\frac{\tor_{\alpha+1}(N)}{\tor_\alpha(N)}=&\frac{\tor_{\alpha+1}(\RG\otimes_RK) \cap N}{\tor_\alpha(\RG\otimes_RK) \cap N}\cong\\
&\cong \frac{(\tor_{\alpha+1}(\RG\otimes_RK) \cap N)+\tor_\alpha(\RG\otimes_RK)}{\tor_\alpha(\RG\otimes_RK)}\leq \frac{\tor_{\alpha+1}(\RG\otimes_RK)}{\tor_\alpha(\RG\otimes_RK)}\end{align*}
is different from zero for finitely many ordinals $\alpha$. Thus, 
$$\Gdim(N)=\sup\{\alpha+1:\tor_{\alpha+1}(N)/\tor_\alpha(N)\neq 0\}=\max\{\alpha+1:\tor_{\alpha+1}(N)/\tor_\alpha(N)\neq 0\}$$
is clearly a successor ordinal.
\end{proof}
\end{enumerate}
%


\noindent
Now, suppose that $\ker(\phi)\neq 0$ and let $\Gdim(\ker(\phi))=\alpha+1$. We want to show that 
$$\phi\restriction_{\tor_{\alpha+1}(M)}:\tor_{\alpha+1}(M)\to \tor_{\alpha+1}(M)$$ is surjective. Indeed, if there is $x\in \tor_{\alpha+1}(M)\setminus \phi(\tor_{\alpha+1}(M))$, it means that there exists $y\in M\setminus\tor_{\alpha+1}(M)$ such that $\phi(y)=x$ (by the surjectivity of $\phi$). This is to say that there is a short exact sequence
$$0\to \ker(\phi)\cap\RG y\to\RG y \to \RG x\to 0\,,$$
with $\Gdim({}_R(\RG y))\gneq \alpha+1\geq \max\{\Gdim({}_R(\ker(\phi)\cap\RG y)),\Gdim({}_R(\RG x))\}$, which contradicts Lemma \ref{pre1}(4). Thus, we have a short exact sequence of left $\RG$-modules
$$0\to \ker(\phi)\to \tor_{\alpha+1} (M)\to \tor_{\alpha+1}(M)\to 0\, .$$
Consider the length function $\ell_\alpha:\lmod R\to \R_{\geq 0}\cup\{\infty\}$ 
described in Subsection \ref{examples} and recall that $\ker(\ell_\alpha)$ is exactly the class of all left $R$-modules with Gabriel dimension $\leq \alpha$. Furthermore,  $\tor_{\alpha+1}(K)$ is a Noetherian module, thus, $\Q_\alpha(\tor_{\alpha+1}(K))$ is a Noetherian object in a semi-Artinian category, that is, an object with finite composition length, for this reason $\ell_\alpha(\tor_{\alpha+1}(K))=\ell(\Q_\alpha(\tor_{\alpha+1}(K)))<\infty$. Using the computations of Example \ref{Bernoulli} and the Addition Theorem, we get
$$\ent_{\ell_\alpha}(\tor_{\alpha+1}(\RG\otimes K)))=\ell_{\alpha}(\tor_{\alpha+1}(K))<\infty\ \ \ \text{ and }\ \ \ \ent_{\ell_\alpha}(\tor_{\alpha+1}(M))=\ent_{\ell_\alpha}(\tor_{\alpha+1}(M))+\ent_{\ell_\alpha}(\ker(\phi))\,.$$ 
Hence, $\ent_{\ell_\alpha}(\ker(\phi))=0$ which, by Proposition \ref{ent>>0}, is equivalent to say that $\ell_\alpha(\ker(\phi))=0$, contradicting the fact that $\Gdim(\ker(\phi))=\alpha+1$.
\end{proof}

In the above proof we made use of the Addition Theorem for the algebraic entropy, which is quite a deep result. We want to underline that if one is only interested in the second part of the statement, that is, stable finiteness of endomorphism rings, then it is sufficient to use the weaker additivity of the algebraic entropy on direct sums, which can be independently verified as an exercise.

\begin{example}\label{libero}
Let $G$ be a free group of rank $\geq2$ and let $\K$ be a field. It is well-known that $\K[G]$ is not left (nor right) Noetherian so we can find a left ideal $_{\K[G]}I\leq \K[G]$ which is not finitely generated. Furthermore, by \cite[Corollary 7.11.8]{Cohn}, $\K[G]$ is a free ideal ring, so $I$ is free. This means that $I$ is isomorphic to a coproduct of the form $\K[G]^{(\N)}$ which is obviously not hopfian. 
\end{example}

Let us conclude this subsection with the following problem:

\begin{problem}\label{problem}
Study the class of finitely generated groups $G$ such that the group algebra $\K[G]$ is hereditarily Hopfian for any skew field $\K$. Is it true that this property characterizes the class of finitely generated amenable groups?
\end{problem}
One could also state an analogous problem including all the possible crossed products $\K\asterisk G$, instead of just the group algebras $\K[G]$.

\subsection{Zero-Divisors}\label{zero_conj}

In this last section of the paper we discuss another classical conjecture due to Kaplansky about group rings connecting it to the theory of algebraic entropy:
\begin{conjecture}[Kaplansky]\label{zero-divisor}
Let $\K$ be a field and $G$ be a torsion-free group. Then $\K[G]$ is a domain.
\end{conjecture}
Some cases of the above conjecture are known to be true but the conjecture is fairly open in general (for a classical reference on this conjecture see for example \cite{passmann}). In most of the known cases, the strategy for the proof is to find an immersion of $\K[G]$ in some division ring. This is clearly sufficient but, in principle, it is a stronger property. To the best of the author's knowledge, the following question remains open: {\em Is it true that $\K[G]$ is a domain if and only if $\K[G]$ is a subring of a division ring?}\\
The above question is known to have positive answer if $G$ is amenable (see \cite[Example 8.16]{Luck}). In this section we provide an alternative argument to answer the above question for amenable groups (in the more general setting of crossed group rings) and we translate the amenable case of Conjecture \ref{zero-divisor} into an equivalent statement about algebraic entropy. This approach is inspired to the work of Nhan-Phu Chung and Andreas Thom \cite{CT}. Indeed, we can prove the following

\begin{theorem}\label{dom=oredom}
Let $\K$ be a division ring and let $G$ be a finitely generated amenable group. For any fixed crossed product $\K\asterisk G$, the following are equivalent:
\begin{enumerate}[\rm (1)]
\item $\K\asterisk G$ is a left (and right) Ore domain; 
\item $\K\asterisk G$ is a domain;
\item $\ent_{\dim}({}_{\K\asterisk G}M)=0$, for every proper quotient $M$ of $\K\asterisk G$;
\item $\Im(\ent_{\dim})=\N\cup\{\infty\}$.
\end{enumerate}
\end{theorem}

Before proving the above theorem we recall some useful properties about Ore domains. We start recalling that a domain $D$ is left Ore if $Dx\cap Dy\neq\{0\}$ for all $x,y\in D\setminus\{0\}$. It can be shown that this is equivalent to say that $D$ is a left flat subring of a division ring. 

\begin{proposition}\label{rank=Ore}
A domain $D$ is left Ore if and only if there is a length function $L\colon \lmod D\to \R_{\geq 0}\cup\{\infty\}$ such that $L(D)=1$.
\end{proposition}
\begin{proof}
If $D$ is left Ore, then $D$ is a flat subring of a division ring $\K$. Then there is an exact functor $\K\otimes_{D}-:\lmod D\to \lmod \K$ which commutes with direct limits. Thus, we can define the desired length function $L$ simply letting $L({}_DM):=\dim_{\K}(\K\otimes_D M)$. 
\\
On the other hand, suppose that there is a length function $L\colon\lmod D\to \R_{\geq 0}\cup\{\infty\}$ such that $L(D)=1$ and choose $x,$ $y\in D\setminus \{0\}$. Since $D$ is a domain, both $Dx$ and $Dy$ contain (and are contained in) a copy of $D$, thus $L(Dx)=L(Dy)=1$. If, looking for a contradiction $Dx\cap Dy=\{0\}$, then
$$1=L(D)\geq L(Dx+Dy)=L(Dx\oplus Dy)=L(Dx)+L(Dy)=2,$$
which is  a contradiction. 
\end{proof}

It is a classical result that any left Noetherian domain is left Ore (see for example \cite[Theorem 1.15 in Chapter 2.1]{McRobson}). By the above proposition we can generalize this result as follows:
\begin{corollary}
A domain with left Gabriel dimension is necessarily left Ore.
\end{corollary}
\begin{proof}
Let $D$ be a domain with left Gabriel dimension. First of all we  verify that $\Gdim ({}_DD)$ is not a limit ordinal. Indeed, if $\Gdim({}_DD)=\lambda$ is a limit ordinal, then $D=\bigcup_{\alpha<\lambda}\tor_\alpha(D)$. This means that, for any non-zero $x\in D$, there exists $\alpha<\lambda$ such that $Dx\in \tor_\alpha(D)$. Choose a non-zero $x\in D$, as $D$ is a domain, there is a copy of $D$ inside $Dx$. Thus, $\Gdim (D)\leq\Gdim(Dx) \leq\alpha$ for some $\alpha<\lambda$, a contradiction.\\
If $\Gdim ({}_DD)=\alpha+1$ for some ordinal $\alpha$, then we can consider the length function
$$\ell_\alpha\colon\lmod D\to\R_{\geq 0}\cup\{\infty\}\ \,,\ \ \ \ell_\alpha(M)=\ell(\Q_\alpha(M))\, .$$
To conclude one has to show that $\ell_\alpha(D)=1$, that is, $\Q_\alpha(D)$ is a simple object. Since $\C_{\alpha+1}/\C_\alpha$ is semi-Artinian, there is a simple subobject  $S$ of  $\Q_\alpha(D)$. Then $\S_\alpha(S)$ is a sub-module of $\S_\alpha \Q_\alpha(D)$. Identify $\S_\alpha(S)$, $\S_\alpha \Q_\alpha(D)$ and $D$ with submodules of $E(D)$, since $D$ is essential in $E(D)$, there is $0\neq x$ such that $x\in \S_\alpha(S)\cap D$, but then $\S_\alpha(S)$ contains an isomorphic copy of $D$. Thus $\Q_\alpha \S_\alpha(S)=S$ contains an isomorphic copy of $\Q_\alpha(D)$, which is therefore simple.
\end{proof}

We can finally prove our result:
\begin{proof}[Proof of Theorem \ref{dom=oredom}]
(1)$\Rightarrow$(2) is trivial while (2)$\Rightarrow$(1) follows by Proposition \ref{rank=Ore} and the fact that the algebraic $\dim$-entropy is a length function on $\lmod{\K\asterisk G}$ such that $\ent_{\dim}({}_{\K\asterisk G}\K\asterisk G)=1$.

\smallskip\noindent
(2)$\Rightarrow$(3). Consider a short exact sequence $0\to {}_{\K\asterisk G}I\to {}_{\K\asterisk G}\K\asterisk G\to {}_{\K\asterisk G}M\to 0$,
with $I\neq 0$. Choose $0\neq x\in I$, then $\K\asterisk Gx\cong \K\asterisk G$, and so $\ent_{\dim}( {}_{\K\asterisk G}M)=\ent_{\dim}({}_{\K\asterisk G}\K\asterisk G)-\ent_{\dim}({}_{\K\asterisk G}I)\leq 1-1=0$.

\smallskip\noindent
(3)$\Rightarrow$(4). Let us show first that for any finitely generate left $\K\asterisk G$-module ${}_{\K\asterisk G}F$, $\ent_{\dim}({}_{\K\asterisk G}F)\in \N$. In fact, choose a finite set of generators $x_1,\dots,x_n$ for $F$ and, letting $F_0=0$ and $F_i=\K\asterisk Gx_1+\dots+\K\asterisk Gx_i$ for all $i=1,\dots,n$, consider the filtration $0\subseteq F_1 \subseteq F_2\subseteq \dots\subseteq F_n=F$. By additivity, 
$$\ent_{\dim}(F)=\sum_{i=1}^n\ent_{\dim}(F_i/F_{i-1})\, .$$
All the modules $F_1/F_0,\dots,F_n/F_{n-1}$ are cyclics (i.e. quotients of $\K\asterisk G$), thus $\ent_{\dim}(F_i/F_{i-1})\in\{0,1\}$ by hypothesis. Hence, $\ent_{\dim}(F)\in \N$. To conclude one argues by upper continuity that the algebraic $\dim$-entropy of an arbitrary left $\K\asterisk G$-module is the supremum of a subset of $\N$, thus it belongs to $\N\cup\{\infty\}$.

\smallskip\noindent
(4)$\Rightarrow$(2). Let $x\in \K\asterisk G$ and consider the short exact sequence 
$$0\to I \to \K\asterisk G\to \K\asterisk Gx\to 0 $$
where $I=\{y\in \K\asterisk G: yx=0\}$. Suppose that $x$ is a zero-divisor, that is, $I\neq 0$ or, equivalently, $\dim ({}_{\K}I)\neq 0$. By Proposition \ref{ent>>0}, $\ent_{\dim}(I)>0$ and, by our assumption, $\ent_L(I)\geq 1$. Hence, using additivity, $\ent_{\dim}(\K\asterisk Gx)=0$. Again by Proposition \ref{ent>>0}, this implies $\dim(\K\asterisk Gx)=0$ and consequently $\K\asterisk Gx=0$, that is, $x=0$. Thus, the unique zero-divisor in $\K\asterisk G$ is $0$.
\end{proof}

\bibliographystyle{alpha}
\bibliography{refs}

\bigskip
Address of the author:

\medskip

Simone Virili - {\tt virili.simone@gmail.com}

Facultad de matem\'aticas, Universidad de Murcia, Campus Espinardo, 30100 Murcia, Spain.

\end{document}